\documentclass[twoside, 12pt]{article}
\usepackage{mathrsfs}
\usepackage{amssymb, amsmath, mathrsfs, amsthm}
\usepackage{graphicx}
\usepackage{color}
\usepackage[top=2.5cm, bottom=2.5cm, left=2.5cm, right=2.5cm]{geometry}
\usepackage{float, caption, subcaption}

\input{epsf}

\newcommand{\bd}{\begin{description}}
\newcommand{\ed}{\end{description}}
\newcommand{\bi}{\begin{itemize}}
\newcommand{\ei}{\end{itemize}}
\newcommand{\be}{\begin{enumerate}}
\newcommand{\ee}{\end{enumerate}}
\newcommand{\beq}{\begin{equation}}
\newcommand{\eeq}{\end{equation}}
\newcommand{\beqs}{\begin{eqnarray*}}
\newcommand{\eeqs}{\end{eqnarray*}}

\newcommand{\Rmnum}[1]{\expandafter\@slowromancap\romannumeral #1@}

\definecolor{DarkGreen}{rgb}{0.2, 0.6, 0.3}

\newtheorem{theorem}{Theorem}[section]
\newtheorem{conjecture}{Conjecture}
\newtheorem{prop}{Proposition}

\newtheorem{lemma}{Lemma}[section]
\theoremstyle{definition}\newtheorem{definition}{Definition}
\newtheorem{coro}[theorem]{Corollary}

\newtheorem{claim}{Claim}
\newtheorem{remark}{Remark}[section]

\begin{document}
\title{\textbf{Forbidden rainbow configurations in a family of edge-disjoint uniform hypergraphs}\footnote{\scriptsize Supported by the National Natural Science Foundation of China No. 12201375. Email: lp-math@snnu.edu.cn}
}

\author{\small Ping Li\\
{\small School of Mathematics and Statistics,}\\
{\small  Shaanxi Normal University, Xi'an, Shaanxi 710119, China}
}
\date{}
\maketitle

\begin{abstract}
The Ruzsa-Szemer\'{e}di $(6,3)$-problem can be equivalently stated as determining the maximum number of edge-disjoint triangles on $n$ vertices such that no triangle is formed by edges from three distinct triangle-copies.
Gowers and Janzer extended this problem by establishing an analogous result for complete graphs.
A natural generalization of the two results, first introduced by Imolay,  Karl,  Nagy and V\'{a}li, asks for the maximum number of edge-disjoint copies of a graph $F$ on $n$ vertices such that no copy of $G$ is formed by edges originating from distinct $F$-copies. 
This maximum number, denoted by $ex_F(n,G)$,  is called the {\em $F$-multicolor Tur\'{a}n number} of $G$.

This paper focuses on the setting of uniform hypergraphs. 
We first prove that for $k$-uniform hypergraphs $\mathcal{G}$ and $\mathcal{F}$, $ex_{\mathcal{F}}(n,\mathcal{G})=o(n^k)$ if and only if there exists a homomorphism from $\mathcal{G}$ to  $\mathcal{F}$.
For degenerate case, we show that  
$ex_{\mathcal{F}}(n,\mathcal{G})=n^{k-o(1)}$ whenever $\mathcal{G}$ contains a $k$-uniform tight triangle.
These results extend the main theorems of [{\em Discrete Math.}, 345 (2022), 112976] and [{\em J. Graph Theory}, 107(3) (2024), 1--13], respectively.
We further establish corresponding supersaturation and blowup statements. 
In the non-degenerate setting, we derive matching lower and upper bounds for 
$ex_{\mathcal{F}}(n,\mathcal{G})$.
We give a necessary and sufficient condition for $ex_{\mathcal{F}}(n,\mathcal{G})$ to fail to attain the upper bound, under the assumption that the extremal graphs for $\mathcal{G}$ are stable. 
As an application, we refine a result due to Imolay,  Karl, Nagy and V\'{a}li. Furthermore, we completely characterize 
$\mathcal{F}$ for which $ex_{\mathcal{F}}(n,\mathcal{G})$ does not attain the upper bound when $\mathcal{G}$ is one of the three special intersecting graphs: Fano plane, extended triangle and $r$-book of $r$-edges with $r=3,4$. 
\\[0.2cm]
\end{abstract}

\section{Introduction}

In 1973, Brown, Erd\H{o}s and S\'{o}s \cite{BROWN,BROWN-2} initiated the study of the function $f_r(n,v,e)$, defined as  the maximum number of edges in an $n$-vertex $r$-uniform hypergraph such that any $v$ vertices span fewer than $e$ edges. 
In particular, if $e=\binom{v}{r}$, then  $f_r(n,v,e)=ex_r(n,K_v^{(r)})$ is the Tur\'{a}n number of the $v$-vertex complete $r$-uniform hypergraph $K_v^{(r)}$ (more generally, for an $r$-uniform hypergraphs $\mathcal{G}$, $ex_r(n,\mathcal{G})$ denotes the Tur\'{a}n number of $\mathcal{G}$, and $EX_r(n,\mathcal{G})$ denotes the set of $n$-vertex extremal graphs).
They proved that 
$$f_r(n,e(r-k)+k,e)=\Theta(n^k)$$
when $2\leq k<r$ and $e\geq 3$.
This suggests a natural but difficult problem of computing the asymptotic value of 
$$f_r(n,e(r-k)+k+1,e).$$
Establishing the value of $f(n,6,3)$  is known as $(6,3)$-problem (the case when $k+1=r=e=3$), which was resolved by Ruzsa and szemer\'{e}di \cite{6-3} in 1976 by proving 
$$n^2e^{-O(\sqrt{\log n})}=n^{2-o(1)}< f(n,6,3)=o(n^2).$$
Alon and Shapira \cite{Alon-Shapira} generalized the $(6,3)$-problem by determining the asymptotic value of $f_r(n,e(r-k)+k+1,e)$ when $e=3$ and $2\leq k<r$.
\begin{theorem}[Alon and Shapira, \cite{Alon-Shapira}]\label{thm-Alon-Shapira}
For any fixed $2\leq k<r$, we have 
$$n^{k-o(1)}<f_r(3(r-k)+k+1,3)=o(n^k).$$
\end{theorem}
\begin{remark}\label{remark-1}
In fact, Alon and Shapira \cite{Alon-Shapira} obtained a more stronger result: Theorem \ref{thm-Alon-Shapira} holds for $r$-uniform hypergraphs that any two edges intersect in at most $k-1$ vertices.
\end{remark}
They conjectured that for the case $e\geq 4$, the function $f_r(n,e(r-k)+k+1,e)$ exhibits asymptotic behavior similar to the case $e=3$.
\begin{conjecture}[Alon and Shapira, \cite{Alon-Shapira}]\label{conj-Alon-Shapira}
For any fixed $2\leq k<r$ and $e\geq 4$, we have 
$$n^{k-o(1)}<f_r(n,e(r-k)+k+1,e)=o(n^k).$$
\end{conjecture}

Let us return to the discussion of the $(6,3)$-problem.
Note that if two edges $e_1$ and $e_2$ share exactly two vertices, then $e_1\cup e_2$ forms a connected component. This implies that, for sufficiently large $n$, the $(6,3)$-problem can be equivalently studied in the setting of linear $3$-uniform hypergraph.
Consequently, an equivalent formulation of the  $(6,3)$-problem is 
to determine the maximum number of edges in an $n$-vertex linear $3$-uniform hypergraph that contains no Berge-triangle.
It is also worth noting that if each $3$-edge is interpreted as a triangle, then the $(6,3)$-problem can be further equivalently reformulated as following two statements: 
\begin{enumerate}
    \item [({\bf E}1).] Find the maximum number of  triangles in an $n$-vertex graph such that any two triangles share at most one common vertex.
    \item [({\bf E}2).] Embed the maximum number of edge-disjoint triangles into an $n$-vertex set such that no now triangle is formed by edges originating from three distinct triangle-copies.
\end{enumerate}

Gowers and Janzer \cite{Gowers} introduced a generalization of the 
$(6,3)$-problem, building on its equivalent formulation (E1).
They posed the following question: for integers $1\leq k<s$, what is the maximum number of copies of  $K_s$  in an $n$-vertex graph such that every $K_k$ is contained in at most one copy of $K_s$.
Note that the $(6,3)$-problem corresponds to the case $k=2$ and $s=3$.
Gowers and Janzer \cite{Gowers} resolved this problem by employing a probabilistic method based on random graph models embedded in high-dimensional Euclidean spaces.
\begin{theorem}[Gowers and Janzer, \cite{Gowers}]\label{thm-Gowers}
For each $1\leq k<s$ and positive $n$, there is a graph $G$ on $n$ vertices with $n^ke^{-O(\sqrt{\log n})}=n^{k-o(1)}$ copies of $K_s$ such that every $K_k$ is contained in at most one $K_s$.
\end{theorem}
In Theorem \ref{thm-Gowers}, for $2\leq k<s$, the graph $G$ naturally induces an $k$-uniform hypergraph $\mathcal{G}$ with $V(G)=V(\mathcal{G})$ and $E(G)=\{S:G[S]\mbox{ is a copy of }K_k\}$.
Observe that for any $S\in V(G)$ of size $s$, the induced graph $G[S]$ is a copy $K_s$ if and only if $\mathcal{G}[S]$ forms a copy of $K_s^{(k)}$.
According to Theorem \ref{thm-Gowers}, any two copies of $K_s^{(k)}$ in $\mathcal{G}$ are edge-disjoint, and it further implies the following result.
\begin{coro}\label{colo-1}
For $2\leq k<s$, let $N_{k,s}(n)$ denote the maximum number of edge-disjoint copies of $K_s^{(k)}$ on $n$ vertices such that 
there is no $K_s^{(k)}$  whose edges come from different $K_s^{(k)}$-copies. Then $N_{k,s}(n)\geq n^ke^{-O(\sqrt{\log n})}=n^{k-o(1)}$.
\end{coro}

Motivated by the equivalent formulation ({\bf E}2) of the $(6,3)$-problem and  Corollary \ref{colo-1}, a natural question arises: for two $k$-uniform hypergraphs $\mathcal{F}$ and $\mathcal{G}$, what is the maximum number of edge-disjoint $\mathcal{F}$-copies on $n$ vertices such that no copy of $\mathcal{G}$ is formed by edges originating from distinct $\mathcal{F}$-copies.
We denote by $ex_{\mathcal{F}}(n,\mathcal{G})$ the above maximum value and refer to it as the {\em  $\mathcal{F}$-multicolor Tur\'{a}n number} of $\mathcal{G}$.
The term ``$\mathcal{F}$-multicolor" stems from an alternative interpretation: if each such $\mathcal{F}$-copy is regarded as a monochromatic graph with a unique color, then $ex_{\mathcal{F}}(n,\mathcal{G})$ represents the maximum number of monochromatic copies of $\mathcal{F}$ on $n$ vertices  under the constraint that no rainbow copy of $\mathcal{G}$ appears.
This variant of Tur\'{a}n-type problem was first introduced by Imolay,  Karl,  Nagy and V\'{a}li \cite{IKNV} in the setting of simple graphs.

Note that if $k=2$ and $\mathcal{F}=\mathcal{G}=K_3$, then the problem reduces to the $(6,3)$-problem, implying $ex_{K_3}(n,K_3)=n^{2-o(1)}$; if $3\leq k<s$, then Corollary \ref{colo-1} indicates $ex_{K_s^{(k)}}(n,K_s^{(k)})\geq n^{k-o(1)}$.
In general, for an integer $k\geq 3$ and a simple graph $G$, the value $ex_{K_k}(n,G)$ represents the linear Tur\'{a}n number of Berge-$G$, denoted by $ex_k^{lin}(n,\mbox{Berge-}G)$, which is defined as the maximum number of edges in an $n$-vertex Berge-$G$-free linear $k$-unifrom hypergraph.
For general simple graphs $F$ and $G$, one foundational result established by Imolay,  Karl,  Nagy and V\'{a}li \cite{IKNV} is  stated as follows. 
\begin{theorem}[Imolay,  Karl,  Nagy and V\'{a}li, \cite{IKNV}]\label{F-density}
For two nonempty simple graphs $G$ and $F$, $ex_F(n,G)=\Theta(n^2)$ if and only if there is no homomorphism from $G$ to $F$.
\end{theorem}

Authors in \cite{IKNV} also showed that if $\chi(G)>\chi(F)$, then \begin{align}\label{eq-intro-1}
ex_F(n,G)=ex(n,G)/e(F)+o(n^2).
\end{align}
This motivates the study of $ex_F(n,G)$ in the case where $\chi(G)\leq\chi(F)$ and no homomorphism exists from $G$ to $F$.
Among all such graph pairs, the simplest example is $F=C_5$ and $G=C_3$.
Kov\'{a}cs and Nagy \cite{C5-C3-1} showed that $ex_{C_5}(n,K_3)=(1+o(1))n^2/25$. Balogh, Liebenau,  Mattos and Morrison \cite{C5-C3-2} refined the result by showing that for any $\delta>0$, $ex_{C_5}(n,K_3)\leq (n^2+3n)/25+\delta n$ holds for sufficiently large $n$.
For the degenerate case, Kov\'{a}cs and Nagy \cite{C5-C3-1} established the following result, which is a new generalization of the $(6,3)$-problem  $ex_{K_3}(n,K_3)$.
\begin{theorem}[Kov\'{a}cs and Nagy, \cite{C5-C3-1}]\label{F-dege}
If $F,G$ are simple graphs such that $G$ contains a triangle and there is a homomorphism from $G$ to $F$, then $ex_F(n,G)=n^{2-o(1)}$.
In particular, $ex_{K_r}(n,K_s)=n^{2-o(1)}$ if $3\leq s<r$.
\end{theorem}

This paper is primarily concerned with the uniform hypergraphs.
To avoid ambiguity, we refer to each monochromatic $\mathcal{F}$ in the definition of $\mathcal{F}$-multicolor Tur\'{a}n number as an {\em $\mathcal{F}$-copy}, whereas the term {\em a copy of $\mathcal{F}$}  denotes a subgraph that is isomorphic to $\mathcal{F}$.
The first work in this paper is to extend Theorem \ref{F-density} to uniform hypergraphs, as stated below.

\begin{theorem}\label{main-1}
For two $k$-uniform hypergraphs $\mathcal{G}$ and $\mathcal{H}$, $ex_\mathcal{F}(n,\mathcal{G})=\Theta(n^k)$ if and only if there is no homomorphism from $\mathcal{G}$ to $\mathcal{F}$.
\end{theorem}

For a $k$-uniform hypergraph $\mathcal{G}$ and an integer $\ell$, the {\em $\ell$-blowup} of $\mathcal{G}$, denoted by $\mathcal{G}(\ell)$, is defined by replacing each vertex $x$ of $\mathcal{G}$ with $\ell$ copies $x^1,\ldots,x^\ell$, and each edge $\{x_1,\ldots,x_k\}$ of $\mathcal{G}$ with $\ell^k$ edges $\{\{x_1^{a_1},\ldots,x_k^{a_k}\}:a_i\in[\ell]\mbox{ for each }i\in[k]\}$. 
The following two results imply that the supersaturation phenomenon and the blowup phenomenon also exist for $\mathcal{F}$-multicolor number.

\begin{prop}\label{super-sat}
Let $\mathcal{G}$ and $\mathcal{H}$ be two $k$-uniform hypergraphs. For any $\epsilon\in(0,1)$ and positive integer $s$, there exist $\eta\in(0,1)$ and an integer $N$, such that if $n>N$ and $\mathcal{R}$ is a $k$-uniform hypergraph consisting of  $ex_\mathcal{F}(n,\mathcal{G})+\epsilon n^k$  $\mathcal{F}$-copies, then $\mathcal{R}$ contains at least $\eta n^{v(\mathcal{G}(s))}$ copies of rainbow $\mathcal{G}(s)$. 
\end{prop}
 
\begin{theorem}\label{main-2}
Let $\mathcal{G}$ and $\mathcal{H}$ be two $k$-uniform hypergraphs. Then for any integer $s$, $ex_\mathcal{F}(n,\mathcal{G}(s))=ex_\mathcal{F}(n,\mathcal{G})+o(n^k)$.
\end{theorem}

Let $[n]$ denote the set $\{1,2,\ldots,n\}$ and let $[a,b]_{\mathbb{N}}$ denote the set of integers in the interval $[a,b]$.
For $2\leq s\leq k+1$, denote by $H^{(k)}_s$ the unique $k$-uniform hypergraph on $k+1$ vertices with $s$ edges.
In particular, $H_3^{(k)}$ is known as the {\em $k$-uniform tight triangle}.
We extend Theorem \ref{F-dege} as follows. 

\begin{theorem}\label{thm-triangle}
Suppose that $\mathcal{G}$ and $\mathcal{F}$ are $k$-uniform hypergraphs, and $\mathcal{G}$ contains a copy of $H_3^{(k)}$. If there is a homomorphism from $\mathcal{G}$ to $\mathcal{F}$, then $ex_\mathcal{F}(n,\mathcal{G})=n^{k-o(1)}$.
In special, $ex_{K_t^{(k)}}(n,K_r^{(k)})=n^{k-o(1)}$ if $t\geq r>k$.
\end{theorem}

\begin{remark}
When considering $H_2^{(k)}$, the value $ex_{K_t^{(k)}}(n,H_2^{(k)})$ deeply relates to {\em $L$-intersecting} $k$-uniform hypergraphs.
Formally, for an integer $k>1$ and a set $L\subseteq\{0,1,\ldots,k-1\}$, a $k$-uniform hypergraph is said to be $L$-intersecting if any two edges $e_1$ and $e_2$ satisfy $|e_1\cap e_2|\in L$. The special case $L=\{1\}$ corresponds to the Fisher's inequality, and the special case $L=[k-1]$ corresponds to  Erd\H{o}s–Ko–Rado Theorem. Ray-Chaudhuri-Wilson Theorem states that the number of edges in an $n$-vertex $L$-intersecting $k$-uniform hypergraphs do not excess $\binom{n}{|L|}$.
Using the Ray-Chaudhuri-Wilson Theorem and Theorem \ref{thm-Gowers}, we have that for $k\geq 3$, $n^{k-1-o(1)}\leq ex_{K_t^{(k)}}(n,H_2^{(k)})\leq {n\choose k-1}$.
\end{remark}

For the non-degenerate case where there is no homomorphism from $\mathcal{G}$ to $\mathcal{F}$, we give the inequality  \begin{align}\label{lower-upper}
    \frac{n^2}{v(\mathcal{F})^2}+o(n^k)\leq ex_\mathcal{F}(n,\mathcal{G})\leq e(\mathcal{F})^{-1}ex_r(n,\mathcal{G})+o(n^k).
\end{align}
We also verify that $ex_\mathcal{F}(n,\mathcal{G})$ attains the upper bound in Ineq.~(\ref{lower-upper}) for any $k$-part $k$-uniform hypergraph $\mathcal{F}$.

\begin{theorem}\label{thm-r-part}
For $k$-uniform hypergraphs $\mathcal{G}$ and $\mathcal{F}$, if there is no homomorphism form $\mathcal{F}$ to $\mathcal{G}$, then Ineq.~(\ref{lower-upper}) holds.
Moreover, if $\chi(\mathcal{F})=k$, then $ex_\mathcal{F}(n,\mathcal{G})= e(\mathcal{F})^{-1}ex_k(n,\mathcal{G})+o(n^k)$.
\end{theorem}

Note that for the case $k=2$, $ex_\mathcal{F}(n,\mathcal{G})$ attains the upper bound when $\chi(\mathcal{F})<\chi(\mathcal{G})$ (see Eq.(\ref{eq-intro-1})).
We will show that this result does not suitable uniform hypergraphs.
In fact, Eq.(\ref{eq-intro-1}) holds due to the stability result for $G$. 
However, when moving to hypergraphs, much less is known about extremal graphs and their stability, making it difficult to establish a full characterization analogous to Eq.(\ref{eq-intro-1}). 
Nevertheless, under the assumption that the extremal hypergraphs and stability, we can still provide a necessary and sufficient condition for $ex_\mathcal{F}(n,\mathcal{G})$ that does not attain its upper bound in Ineq.~(\ref{lower-upper})  by establishing a relationship between $ex_\mathcal{F}(n,\mathcal{G})$ and the $\mathcal{F}$-packing number in extremal graphs for $\mathcal{G}$.
\begin{enumerate}
    \item [$\bullet$] For an $r$-uniform hypergraph $\mathcal{G}$, we say that the extremal graphs of $\mathcal{G}$ are {\em stable} if for any $\epsilon>0$, there exist an integer $N$ and $\delta>0$ such that for any $n\geq N$ and each $n$-vertex $\mathcal{G}$-free $r$-uniform hypergraph $\mathcal{H}_n$ with $e(\mathcal{H}_n)\geq ex_r(n,\mathcal{G})-\delta n^r$, $\mathcal{H}_n$ can be transformed into an extremal graph in $EX_r(n,\mathcal{G})$ by deleting or adding at most $\epsilon n^r$ edges.
    \item [$\bullet$] The {\em $\mathcal{F}$-packing number of $\mathcal{H}$}, denoted by $\nu_\mathcal{H}(\mathcal{F})$, is the maximum number of edge-disjoint copies of $\mathcal{F}$ in  $\mathcal{H}$.
\end{enumerate}

\begin{prop}\label{prop-chara-upper}
For $k$-uniform hypergraphs $\mathcal{G}$ and $\mathcal{F}$, the following two statements hold. 
\begin{enumerate}
    \item If for any $n$ there exists an 
    $\mathcal{H}_n\in EX_k(n,\mathcal{G})$ such that $\nu_{\mathcal{H}_n}(\mathcal{F})= e(\mathcal{F})^{-1}ex_k(n,\mathcal{G})+o(n^k)$, then $ex_\mathcal{F}(n,\mathcal{G})= e(\mathcal{F})^{-1}ex_k(n,\mathcal{G})+o(n^k)$.
    \item If the extremal graphs of $\mathcal{G}$ are stable, then $ex_\mathcal{F}(n,\mathcal{G})$ does not reach the upper bound of Ineq.~(\ref{lower-upper}) if and only if there exists a real number $\epsilon>0$ such that for infinitely many $n$, every graph $\mathcal{H}_n\in EX_k(n,\mathcal{G})$ satisfies $\nu_{\mathcal{H}_n}(\mathcal{F})\leq e(\mathcal{F})^{-1}ex_k(n,\mathcal{G})-\epsilon n^k$.
\end{enumerate}
\end{prop}

For uniform hypergraphs, we can use Proposition \ref{prop-chara-upper} to verify whether  $$ex_\mathcal{F}(n,\mathcal{G})=e(\mathcal{F})^{-1}ex_k(n,\mathcal{G})+o(n^k)$$ 
for some given pairs $\mathcal{F}$ and $\mathcal{G}$.
Here, we focus primarily on the case where   
$\mathcal{G}$ is one of the following three specific intersecting hypergraph graphs: the Fano plane $\mathbb{F}$, the extended $2k$-uniform triangle $\mathcal{C}_3^{(2k)}$ and the  $k$-book $\mathcal{B}_k^{(k)}$ of $k$ pages (see Figure \ref{3-figs}).
We call a $2k$-uniform hypergraph {\em odd} if its vertex-set can be partitioned into two sets such that every edge intersects both parts in an odd number of vertices.
A $k$-uniform hypergraph is called  {\em bipartite} if its vertex-set can be partitioned into two sets such that every edge intersects both parts in at least one vertex. An {\em independent set } in a hypergraph is a vertex subset in which no two vertices lie in the same edge.
For a vertex set $S\subseteq V(\mathcal{H})$, let $\mathcal{H}-S$ denote the graph obtained from  $\mathcal{H}$ by deleting vertices in $S$ and all edges incident to any vertex in $S$.

\begin{figure}[h]
    \centering
    \includegraphics[width=300pt]{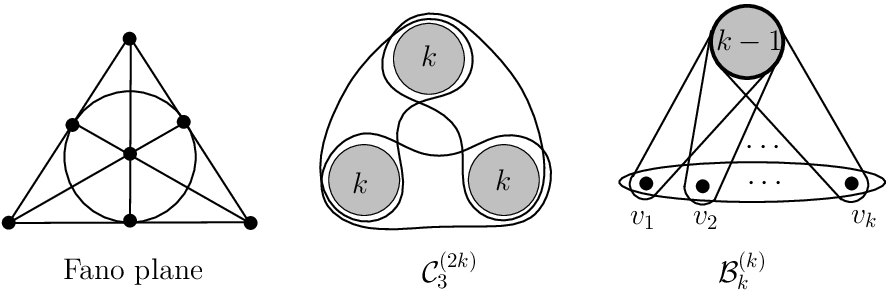}\\
    \caption{The left graph is known as the {\em Fano plane}  (the circle and segments represent edges), denoted by $\mathbb{F}$. The middle graph $\mathcal{C}_3^{(2k)}$ is called the {\em extended $2r$-uniform triangle}, which is a $2k$-uniform hypergraph with $V(\mathcal{C}_3^{(2k)})=[3k]$ and $E(\mathcal{C}_3^{(2k)})=\{[2k],[k+1,3k]_{\mathbb{N}},[3k]-[k+1,2k]_{\mathbb{N}}\}$. The right graph $\mathcal{B}_k^{(k)}$ is called the {\em $k$-book of $k$ pages}, which is a $k$ uniform hypergraph with $V(\mathcal{B}_k^{(k)})=[k-1]\cup\{v_i:i\in[k]\}$ and $E(\mathcal{B}_k^{(k)})=\{[k-1]\cup\{v_i\}:i\in[k]\}\cup\{v_i:i\in[k]\}$.}\label{3-figs}
\end{figure}

\begin{theorem}\label{thm-fano}
For $r$-uniform hypergraphs $\mathcal{F}$ and $\mathcal{G}\in\{\mathbb{F},\mathcal{C}_3^{(2k)},\mathcal{B}_3^{(3)},\mathcal{B}_4^{(4)}\}$, if $\mathcal{F}$ and $\mathcal{G}$ satisfy one of the following statements, then $ex_\mathcal{F}(n,\mathcal{G})$ fails to achieve the upper bound in Ineq.~(\ref{lower-upper}); otherwise, it attains this upper bound.
\begin{enumerate}
    \item $\mathcal{G}=\mathbb{F}$ and $\mathcal{F}$ is not bipartite.
    \item $\mathcal{G}=\mathcal{B}_3^{(3)}$ and for every independent set $S$ of $\mathcal{F}$, $\mathcal{F}-S$ is not an empty graph.
    \item $\mathcal{G}=\{\mathcal{C}_3^{(2k)},\mathcal{B}_4^{(4)}\}$ and $\mathcal{F}$ is not odd.
\end{enumerate}
\end{theorem}
\begin{remark}
    We choose $\mathcal{F}$ as a complete uniform hypergraph $K_s^{(k)}$. It is clear that if $\mathcal{G}=\mathbb{F}$ and $s\in\{5,6\}$, or $\mathcal{G}=\mathcal{C}_3^{(2k)}$ and $s\in[k+1,2k-1]_{\mathbb{N}}$, or $\mathcal{G}=\mathcal{B}_4^{(4)}$ and $r\in\{5,6\}$, 
    or $\mathcal{G}=\mathcal{B}_3^{(3)}$ and $r=4$,
then $ex_{K_s^{(k)}}(n,\mathcal{G})$ fails to achieve the upper bound in Ineq.~(\ref{lower-upper}).
\end{remark}

The following corollary refines a result from \cite{IKNV} by providing a complete characterization of simple graph pairs $G$ and $F$ for which $ex_F(n,G)$ satisfies Eq. \eqref{eq-intro-1}.
\begin{coro}\label{coro-refine-simple}
    For two simple graphs $F$ and $G$ with $\chi(G)\geq 3$, 
    $$ex_F(n,G)=\frac{\chi(G)-2}{2e(F)(\chi(G)-1)}n^2+o(n^2)$$ if and only if $\chi(G)>\chi(F)$.
\end{coro}


The proof of Theorem \ref{main-1} and Proposition \ref{super-sat} use tools of Hypergraph Regularity Lemma and Counting Lemma, and some results on fractional packing number  in hypergraphs, which will be introduced in Sections \ref{section-reg} and \ref{sec-prem}. In Section \ref{sec-prem}, we also present some key lemmas for our main proofs.  In Sections  \ref{section-thm-1}, we give a proof of Theorem \ref{main-1}.
In Section \ref{section-thm-2}, we prove Proposition \ref{super-sat} and Theorem \ref{main-2}, and finally we will prove all other results in Section \ref{section-thm-3}.

\section{The Hypergraph Regularity Lemma and Counting Lemma}\label{section-reg}

Hypergraph Regularity Lemma \cite{RJ} and Counting Lemma \cite{NRS} are powerful tools for analyzing dense uniform hypergraphs.
We refer to papers \cite{RSST,RS} for the notations, terminology and definition modes involved in the two lemmas.
We first introduce some basic notation as follows.

For a uniform hypergraph $\mathcal{G}$, let $V(\mathcal{G})$ denote the vertex-set of $\mathcal{G}$ and let $|V(\mathcal{G})|=v(\mathcal{G})$. Let $E(\mathcal{G})$ denote the edge set of $\mathcal{G}$, and let $e(\mathcal{G})$ or $|\mathcal{G}|$ denote the size of $E(\mathcal{G})$.
Sometimes, $\mathcal{G}$ also represents the edge set of itself. 
For a set of $V$, let ${V\choose k}$ denote the set consisting of all $k$-vertex subsets of $V$. 
For pairwise disjoint nonempty vertex-sets $V_1,V_2,\ldots,V_{\ell}$, let $K_{\ell}^{(k)}(V_1,V_2,\ldots,V_{\ell})$ denote the {\em complete $\ell$-partite $k$-uniform hypergraph}, whose edge set is $\{e\subseteq \bigcup_{i\in[\ell]}V_i:|e|=k\mbox{ and }|e\cap V_i|\leq 1\mbox{ for each }i\in[\ell]\}$.
If $|V_1|=|V_2|=\ldots=|V_\ell|=p$, then we also write
$K_{\ell}^{(k)}(V_1,V_2,\ldots,V_{\ell})$ as $K_{\ell}^{(k)}(p)$, which denotes the $p$-blowup of the complete $k$-uniform hypergraph $K_{\ell}^{(k)}$.
We call a spanning sub-hypergraph $\mathcal{H}$ of $K_{\ell}^{(k)}(V_1,V_2,\ldots,V_{\ell})$ an {\em $(\ell,k)$-hypergraph}.
In particular, we define $(\ell,1)$-hypergraph as a vertex-set $V$ along with a vertex partition $V_1,V_2,\ldots,V_{\ell}$ of $V$.
For simplicity, for a subset $I\subseteq [\ell]$, we use $\mathcal{H}_I$ or $\mathcal{H}[\bigcup_{i\in I}V_i]$ to denote the sub-hypergraph of $\mathcal{H}$ induced by $\bigcup_{i\in I}V_i$.

For any $k$-uniform hypergraph $\mathcal{H}$ and an integer $j\geq k$, we use $\mathcal{K}_j(\mathcal{H})$ to denote the set of all copies of $K_j^{(k)}$ in $\mathcal{H}$ 
(sometimes, we also regard $\mathcal{K}_j(\mathcal{H})$ as a $j$-uniform hypergraph whose edge set is $\{J\subseteq V(\mathcal{H}):|J|=j\mbox{ and }\mathcal{H}[J]\mbox{ is a copy of }K_j^{(k)}\}$).
In particular, if $k=1$ and $j=2$, then $\mathcal{H}$ is a vertex-set $V$ along with a vertex partition $V_1,V_2,\ldots,V_{\ell}$ of $V$, and $\mathcal{K}_j(\mathcal{H})$ is a complete $\ell$-partite graph with parts $V_1,V_2,\ldots,V_{\ell}$.
For a $k$-uniform hypergraph $\mathcal{H}^{(k)}$ and a $(k-1)$-uniform hypergraph $\mathcal{H}^{(k-1)}$, we say $\mathcal{H}^{(k-1)}$ {\em underlies} $\mathcal{H}^{(k)}$ if $\mathcal{H}^{(k)}$ is a subgraph of $\mathcal{K}_k(\mathcal{H}^{(k-1)})$. 
For a $k$-uniform hypergraph $\mathcal{H}$ and vertex partition $V_1,V_2,\ldots, V_{\ell}$ of $V(\mathcal{H})$, a sub-hypergraph $\mathcal{F}$ of $\mathcal{H}$ is {\em crossing} with respect to the vertex partition if $|V(\mathcal{F})\cap V_i|\leq 1$  for each $i\in[\ell]$.

\subsection{Complex and regularity}

We first introduce the  density (or relative density in some papers) of a $k$-uniform hypergraph with respect to a set of $(\ell,k-1)$-hypergraphs.

\begin{definition}[Density]
Suppose that $\mathcal{G}$ is a $k$-uniform hypergraph and $\mathcal{S}=\{\mathcal{F}_1,\ldots,\mathcal{F}_r\}$ is a collection of not necessary disjoint $(k,k-1)$-hypergraphs with the same vertex-set as $\mathcal{G}$. We define the {\em density of $\mathcal{G}$ with respect to $\mathcal{S}$} by 
$$
d(\mathcal{G}|\mathcal{S})=\left\{
\begin{array}{ll}
\frac{\left|\mathcal{G}\cap \bigcup_{i\in[r]}\mathcal{K}_k(\mathcal{F}_i)\right|}
{\left|\bigcup_{i\in[r]}\mathcal{K}_k(\mathcal{F}_i)\right|}, &  \mbox{ if } \left|\bigcup_{i\in[r]}\mathcal{K}_k(\mathcal{F}_i)\right|>0,\\
0, & \mbox{ otherwise}.\\
\end{array}
\right.
$$
Specifically, if $\mathcal{S}=\{\mathcal{F}_1\}$, then we use $d(\mathcal{G}|\mathcal{F}_1)$ to denote $d(\mathcal{G}|\mathcal{S})$, and call it the {\em density of $\mathcal{G}$ with respect to $\mathcal{F}_1$}. 
\end{definition}

Throughout the paper, our discussion are centered around a sequence of graphs, called complex, which is defined as follows.

\begin{definition}[Complex]
Let $\ell\geq j\geq 1$ be integers. An {\em $(\ell,j)$-complex} $\boldsymbol{\mathcal{H}}$ is a collection of $(\ell,i)$-hypergraphs $\{\mathcal{H}^{(i)}\}_{i=1}^j$ such that
\begin{enumerate}
  \item the $(\ell,1)$-hypergraph $\mathcal{H}^{(1)}$ is a vertex-set with partition $V_1,V_2,\ldots,V_\ell$;
  \item for each $i\in[j-1]$, $\mathcal{H}^{(i)}$ underlies $\mathcal{H}^{(i+1)}$.
\end{enumerate}
\end{definition}

We now define the regularity of an $i$-uniform hypergraph with respect to an $(\ell,i-1)$-hypergraph.

\begin{definition}[$(\epsilon,d,r)$-regular]
Let $\epsilon,d\in[0,1]$ be real numbers with $0<\epsilon <d\leq 1$ and let $r$ be a positive integer. Let $\mathcal{G}$ be an $i$-uniform hypergraph and $\mathcal{H}^{(i-1)}$ be an $(\ell,i-1)$-hypergraph with $V(\mathcal{G})=V(\mathcal{H}^{(i-1)})$. We say that {\em $\mathcal{G}$ is $(\epsilon,d,r)$-regular with respect to $\mathcal{H}^{(i-1)}$} if for any $i$-element subset $I$ of $[\ell]$ the following holds: for any sub-hypergraphs $\mathcal{F}_1,\ldots,\mathcal{F}_r$ of $\mathcal{H}^{(i-1)}_I$, if  
$|\bigcup_{j\in[r]}\mathcal{K}_i(\mathcal{F}_j)|\geq \epsilon|\mathcal{K}_i(\mathcal{H}^{(i-1)}_I)|$, then we have that
$d-\epsilon\leq d(\mathcal{G}_I|\{\mathcal{F}_1,\ldots,\mathcal{F}_r\})\leq d+\epsilon$.
In particular, we make a distinction between the following two special cases.
\begin{enumerate}
  \item if $r=1$, we simply say $\mathcal{G}$ is {\em $(\epsilon,d)$-regular with respect to $\mathcal{H}^{(i-1)}$};
  \item we say $\mathcal{G}$ is {\em $(\epsilon,r)$-regular with respect to $\mathcal{H}^{(i-1)}$} if $\mathcal{G}$ is $(\epsilon,d,r)$-regular with respect to $\mathcal{H}^{(i-1)}$ for some $d\in[0,1]$.
\end{enumerate}
\end{definition}

We extend the notion of regularity to complexes as follows.

\begin{definition}[$(\boldsymbol{\epsilon},\mathbf{d},r)$-regular complex]
Let $k,r$ be positive integers, and let $\boldsymbol{\epsilon}=(\epsilon_2,\ldots,\epsilon_k)$ and $\mathbf{d}=\{d_2,\ldots,d_k\}$ be vectors with $0<\epsilon_i\leq d_i\leq 1$ for $2\leq i\leq k$.
$\boldsymbol{\mathcal{H}}=\{\mathcal{H}^{(i)}\}_{i=1}^k$ is an {\em $(\boldsymbol{\epsilon},\mathbf{d},r)$-regular complex} if
\begin{enumerate}
  \item $\mathcal{H}^{(2)}$ is $(\epsilon_2,d_2)$-regular with respect to $\mathcal{H}^{(1)}$;
  \item for $2\leq i\leq k-1$, $\mathcal{H}^{(i+1)}$ is $(\epsilon_{i+1},d_{i+1},r)$-regular with respect to $\mathcal{H}^{(i)}$.
\end{enumerate}
\end{definition}

\subsection{Partitions}

In this subsection, we introduce a family of partitions on a vertex-set $V$ and some related notions.
Unlike the Regularity Lemma \cite{simple-reg}, the partitions in the Hypergraph Regularity Lemma \cite{RJ} extend beyond vertex-sets to include the partitioning of edge sets as well.

For a vertex $V$ and its partition $V=V_1\cup V_2\cup\ldots\cup V_{\ell}$, we define a family of partitions $\mathscr{P}=\{\mathscr{P}^{(i)}\}_{i=1}^{k-1}$ such that $\mathscr{P}^{(1)}$ is the vertex-set $V$ with the partition $V=V_1\cup V_2\cup\ldots\cup V_{\ell}$ and  $\mathscr{P}^{(i)}$ is an edge partition of $K_{\ell}^{(i)}(V_1,\ldots,V_\ell)$ for each $2\leq i\leq k-1$. 
Moreover, each partition class $\mathcal{P}^{(i)}$ of $\mathscr{P}^{(i)}$ is an $(i,i)$-hypergraph. 
It is particularly important to emphasize that there exists an additional relationship between each pair of partitions $\mathscr{P}^{(i)}$ and $\mathscr{P}^{(i-1)}$, which we refer to as cohesive relationship, and we will primarily interpret the relationship in the following discussion.

We first introduce some notations. For each $J\in K_{\ell}^{(j)}(V_1,\ldots,V_\ell)$ where $2\leq j\leq k-1$, we denote by $\mathcal{P}^{(j)}(J)$ the partition class of $\mathscr{P}^{(j)}$ such that $J\in \mathcal{P}^{(j)}(J)$.
For $i<j\leq k-1$ and an edge $J\in K_{\ell}^{(j)}(V_1,\ldots,,V_\ell)$, let $\widehat{\mathcal{P}}^{(i)}(J)$ denote the union 
$\bigcup_{I\in{J\choose i}} \mathcal{P}^{(i)}(I)$ (note that each partition class of $\mathscr{P}^{(i)}$ is an $(i,i)$-hypergraph). 
If $i=j-1$, then we call $\widehat{\mathcal{P}}^{(j-1)}(J)$ a {\em $j$-polyad} and define  
$$\mathscr{\widehat{P}}^{(j)}=\{\widehat{\mathcal{P}}^{(j-1)}(J): J\in \mathcal{K}_{\ell}^j(V_1,\ldots,,V_\ell)\}.$$
Note that for each edge $J\in K_{\ell}^{(j)}(V_1,\ldots,V_\ell)$, there is exactly one $j$-polyad  $\widehat{\mathcal{P}}^{(j-1)}(J)\in\mathscr{\widehat{P}}^{(j)}$ such that $\widehat{\mathcal{P}}^{(j-1)}(J)$ underlies $J$. 
Hence, $\{\mathcal{K}_j(\widehat{\mathcal{P}}^{(j-1)}):\widehat{\mathcal{P}}^{(j-1)}\in \mathscr{\widehat{P}}^{(j)}\}$ is a partition of $K_{\ell}^{(j)}(V_1,\ldots,,V_\ell)$.
The notation cohesive is defined as follows.

\begin{definition}[cohesive]
For $1\leq i\leq k-1$, we call $\mathscr{P}^{(i)}$ and $\mathscr{P}^{(i-1)}$ are {\em cohesive} if
each partition class of $\mathscr{P}^{(i)}$ is a subset of  $\mathcal{K}_j(\widehat{\mathcal{P}}^{(j-1)})$ for some $\widehat{\mathcal{P}}^{(j-1)}\in \mathscr{\widehat{P}}^{(j)}$.
\end{definition}

To facilitate the introduction of a family of partitions $\mathscr{P}$ and the concept of regularity in the Hypergraph Regularity Lemma, we need the following definitions.

\begin{definition}[a family of partition $\mathscr{P}(k-1,\mathbf{a})$]
Let $V$ be a vertex-set of order $n$ and $\mathbf{a}=(a_1,a_2,\ldots,a_{k-1})$ be a vector of positive integers. If partitions $\mathscr{P}=\{\mathscr{P}^{(i)}\}_{i=1}^{k-1}$ satisfies 
\begin{enumerate}
  \item $\mathscr{P}^{(1)}$ has $a_1$ partition classes, that is, $\mathscr{P}^{(1)}$ partitions $V$ into $a_1$ subsets,
  \item for $i\in[k-2]$, $\mathscr{P}^{(i+1)}$ and $\mathscr{P}^{(i)}$ are cohesive, and
  \item for $2\leq j\leq k-1$, 
  $|\{\mathcal{P}^{(j)}\in \mathscr{P}^{(j)}: \mathcal{P}^{(j)}\in \mathcal{K}_j(\widehat{\mathcal{P}}^{(j-1)})\}|\leq a_j$ for each  $\widehat{\mathcal{P}}^{(j-1)}\in \widehat{\mathscr{P}}^{(j)}$,
\end{enumerate}
then we say $\mathscr{P}$ is {\em a family of partition $\mathscr{P}=\mathscr{P}(k-1,\mathbf{a})$ on $V$}.
For an integer $L$, if $\max\{a_1,\ldots,a_{k-1}\}\leq L$, then we say $\mathscr{P}=\mathscr{P}(k-1,\mathbf{a})$ is {\em $L$-bounded}.
\end{definition}

\begin{remark}\label{relark-com}
Assume that  $\mathscr{P}=\mathscr{P}(k-1,\mathbf{a})=\{\mathscr{P}^{(i)}\}_{i=1}^{k-1}$ is a family of partition on $V$ with $\mathscr{P}^{(1)}=\{V_1,\ldots,V_{a_1}\}$.
Then each $J\in K_{a_1}^{(k)}(V_1,\ldots,V_{a_1})$ uniquely determines a complex 
$$\boldsymbol{\widehat{\mathcal{P}}}^{(k-1)}(J)=\{\widehat{\mathcal{P}}^{(i)}(J)\}_{i=1}^{k-1}.$$
This complex is also uniquely determined by its top layer: the $k$-polyad $\widehat{\mathcal{P}}^{(k-1)}(J)$.
The complex $\boldsymbol{\widehat{\mathcal{P}}}^{(k-1)}(J)$ is called {\em $k$-polyad complex} and let $Com_{k-1}(\mathscr{P})=\{\boldsymbol{\widehat{\mathcal{P}}}^{(k-1)}(J): J\in K_{a_1}^{(k)}(V_1,\ldots,V_{a_1})\}$.
\end{remark}

\subsection{Hypergraph Regularity Lemma and Counting Lemma}

We first introduce Hypergraph Regularity Lemma (see \cite{RJ,RS}).
Before this, we require two additional definitions.

\begin{definition}[$(\mu,\boldsymbol{\epsilon},\mathbf{d},r)$-equitable]\label{definition-uedr}
Let $k,r$ be positive integers and $\mu\in(0,1]$, and let $\boldsymbol{\epsilon}=(\epsilon_2,\ldots,\epsilon_{k-1})$ and $\mathbf{d}=\{d_2,\ldots,d_{k-1}\}$ be vectors with $\epsilon_i,d_i\in[0,1]$ for $2\leq i\leq k-1$.
A family of partition $\mathscr{P}=\mathscr{P}(k-1,\mathbf{a})=\{\mathscr{P}^{(i)}\}_{i=1}^{k-1}$ on $V$ is called {\em $(\mu,\boldsymbol{\epsilon},\mathbf{d},r)$-equitable} if
\begin{enumerate}
  \item $\mathscr{P}^{(1)}=\{V_1,\ldots,V_{a_1}\}$ is a partition of $V$ with $|V_1|\leq |V_2|\leq \ldots\leq |V_{a_1}|\leq |V_1|+1$;
  \item all but at most $\mu{n\choose k}$ many $k$-tuples $K\in {V\choose k}$ belong to $(\boldsymbol{\epsilon},\mathbf{d},r)$-regular complex $\boldsymbol{\widehat{\mathcal{P}}}^{(k-1)}(K)\in Com_{k-1}(\mathscr{P})$.
\end{enumerate}
\end{definition}

\begin{definition}[$(\epsilon,r)$-regular of $\mathcal{H}$ with respect to $\mathscr{P}$]\label{definition-br}
Let $k,r$ be positive integers and $\epsilon\in(0,1]$, and let $\mathcal{H}$ be a $k$-uniform hypergraph on $V$ and $\mathscr{P}=\mathscr{P}(k-1,\mathbf{a})=\{\mathscr{P}^{(i)}\}_{i=1}^{k-1}$  
be a family of partition on $V$ with $\mathscr{P}^{(1)}=\{V_1,\ldots,V_{a_1}\}$.
We say $\mathcal{H}$ is {\em $(\epsilon,r)$-regular with respect to $\mathscr{P}$} if for all but at most $\epsilon {n\choose k}$ edges $K\in K_{a_1}^{(k)}(V_1,\ldots,V_{a_1})$ we have that 
$\mathcal{H}\cap \mathcal{K}_{k}(\widehat{\mathcal{P}}^{(k-1)}(K))$ is $(\epsilon,r)$-regular  with respect to the $k$-polyad $\widehat{\mathcal{P}}^{(k-1)}(K)$.
\end{definition}

\begin{theorem}[Hypergraph Regularity Lemma]\label{regularity-lemma}
For every integer $k\in \mathbb{N}$, all numbers $\epsilon_k,\mu>0$, and any non-negative functions $\epsilon_{k-1}(x_{k-1})$, $\epsilon_{k-2}(x_{k-2},x_{k-1})$, $\ldots$, $\epsilon_{2}(x_2,\ldots,x_{k-1})$ and  $r=r(x_1,x_2,\ldots,x_{k-1})$, there exist integers $N_k$ and $L_k$ such that the following holds. For every $k$-uniform hypergraph $\mathcal{H}$ with $|V(\mathcal{H})|\geq N_k$ there exists a family of partitions $\mathscr{P}=\mathscr{P}(k-1, \mathbf{a})$ on $V(\mathcal{H})$ (say $\mathbf{a}=(a_1,\ldots,a_{k-1})$) and a vector $\mathbf{d}=(d_2,\ldots,d_{k-1})\in(0,1]^{k-2}$ so that
\begin{enumerate}
  \item $\mathscr{P}$ is $(\mu, \boldsymbol{\epsilon}(\mathbf{d}),\mathbf{d},r(a_1,\mathbf{d}))$-equitable and $L_k$-bounded, where $$\boldsymbol{\epsilon}(\mathbf{d})=(\epsilon_{2}(d_2,\ldots,d_{k-1}),\ldots,\epsilon_{k-2}(d_{k-2},d_{k-1}),
      \epsilon_{k-1}(d_{k-1}))$$
      and $r(a_1,\mathbf{d})=r(a_1,d_2,\ldots,d_{k-1})$;
  \item $\mathcal{H}$ is $(\epsilon_k,r(a_1,\mathbf{d}))$-regular with respect to $\mathscr{P}$.
\end{enumerate}
\end{theorem}

\begin{remark}
The vertex partition $\mathscr{P}^{(1)}$ of $\mathscr{P}$ in Theorem \ref{regularity-lemma} can be initialized by refining any given equitable vertex-set partition $W_1,\ldots,W_\ell$ of $V(\mathcal{G})$ (in other words, each part of $\mathscr{P}^{(1)}$ is contained in some $W_i$).
\end{remark}

In order to introduce Counting Lemma \cite{RS}, we need one more definition as follows (here is a variant of the Counting Lemma which can be derived by the main result of \cite{NRS}, and in \cite{PHD} for a proof)
\begin{definition}[$(\boldsymbol{\delta},(\mathbf{d},\geq d),r)$-regular $(m,\mathcal{F})$-complex]\label{count-def}
Let $\mathcal{F}$ be a $k$-uniform hypergraph on vertex-set $[f]$, $d\in(0,1]$ and $r,m$ be positive integers, and let $\mathbf{d}=(d_2,\ldots,d_{k-1})$ and $\boldsymbol{\delta}=(\delta_2,\ldots,\delta_k)$ be victors with $d_i,\delta_j\in(0,1]$ for each $2\leq j\leq k$ and $2\leq i\leq k-1$. 
For a complex  $\boldsymbol{\mathcal{H}}=\{\mathcal{H}^{(i)}\}_{i=1}^k$, we call $\boldsymbol{\mathcal{H}}$ a {\em $(\boldsymbol{\delta},(\mathbf{d},\geq d),r)$-regular $(m,\mathcal{F})$-complex} if 
\begin{enumerate}
  \item $\mathcal{H}^{(1)}=\{V_1,\ldots,V_f\}$ is a vertex  partition with $|V_1|=\ldots=|V_f|=m$, and
  \item for every edge $K\in \mathcal{F}$, the $(k,k)$-complex 
  $\boldsymbol{\mathcal{H}}_K=\{\mathcal{H}^{(i)}_K\}_{i=1}^k$ is a $(\boldsymbol{\delta},(d_2,\ldots,d_{k-1},d_K),r)$-regular for some real number $d_K\in[d,1]$, where $\mathcal{H}^{(i)}_K=\mathcal{H}^{(i)}[\bigcup_{j\in K}V_j]$.
\end{enumerate}
\end{definition}

\begin{theorem}[Counting Lemma]\label{counting-lemma}
For every $k$-uniform hypergraph $\mathcal{F}$ with vertex-set $V(\mathcal{F})=[t]$ and $\nu>0$ the following statement holds. There exist non-negative functions
$$\delta_k(x_k),\delta_{k-1}(x_{k-1},x_k),\ldots,\delta_{2}(x_2,\ldots,x_k),r'(x_2,\ldots,x_k)$$ and $m_0(x_2,\ldots,x_k)$ so that for every choice of $\mathbf{d}=(d_2,\ldots,d_{k-1})\in (0,1]^{k-2}$ and $d_k\in (0,1]$ the following holds. 
If $\mathcal{H}=\{\mathcal{H}^{(1)},...,\mathcal{H}^{(k)}\}$ is a $(\boldsymbol{\delta},(\mathbf{d},\geq d_k),r')$-regular $(m,\mathcal{F})$-complex 
with  
$$\boldsymbol{\delta}=(\delta_{2}(d_2,\ldots,d_k),\ldots,\delta_{k-1}(d_{k-1},d_k),\delta_k(d_k)),$$ $r'=r'(\mathbf{d},d_k)$ and $m\geq m_0(\mathbf{d},d_k)$, 
then $\mathcal{H}^{(k)}$ contains at least
$$(1-\nu)m^t\cdot d_k^{|\mathcal{F}|}\cdot\prod_{i=2}^{k-1} d_i^{|\Delta_i(\mathcal{F})|}$$
copies of $\mathcal{F}$, where $\Delta_i(\mathcal{F})=\{I\in {V(\mathcal{F})\choose i}:\mbox{there is an edge }K\in \mathcal{F}\mbox{ such that }I\subseteq K\}$. 
Furthermore, we may assume that the function $m_0(x_2,...,x_k)$ is non-increasing in every variable.
\end{theorem}

\section{Preliminaries}\label{sec-prem}

The application of the Counting Lemma requires the relationship between $\mathbf{d}$ and $L_k$ established in Theorem \ref{regularity-lemma}. 
In fact, the relationship has been explored in \cite{RS} (see Claim 6.1).
However, since our results depend on the choice of parameters, and for the convenience of reading, we restate the proof here.
\begin{lemma}\label{relation-d-L}
In Theorem \ref{regularity-lemma}, if $\epsilon_i\leq d_i$ for each $2\leq i\leq k-1$ and $\mu<8/9$, then $d_j>L_k^{-2^{k-1}}$ for each $2\leq i\leq k-1$.
\end{lemma}
\begin{proof}
Assume that the family of partitions $\mathscr{P}=\mathscr{P}(k-1, \mathbf{a})$ in Theorem \ref{regularity-lemma} is defined on an $n$-vertex-set $V$.
Suppose to the contrary that $L_k^{-2^{k-1}}\geq d_j$ for some $2\leq j\leq k-1$.
For simplicity, we use $\epsilon_i$ to denote $\epsilon_i(d_i,\ldots,d_{k-1})$ for each $2\leq i\leq k-1$.
We count the number of $j$-tuples, which are either not crossing with respect to $\mathscr{P}^{(1)}$ or does not belong to a $((\epsilon_2,\ldots,\epsilon_{j-1}),(d_2,\ldots,d_{j-1}),r(a_1,\mathbf{d}))$-regular $(j,j-1)$-complex  (called them {\em bad $j$-tuples}).

Note that the number of $(j,j-1)$-polyads is at most 
$${a_1\choose j}\cdot\prod_{i=2}^{j-1}a_i^{j\choose i}\leq L_k^{\sum_{i=1}^{j-1}{j\choose i}}/j!=L_k^{2^j-2}/j!\leq L_k^{2^{k-1}-2}/j!.$$
Since $\mathscr{P}$ is 
$(\mu, \boldsymbol{\epsilon}(\mathbf{d}),\mathbf{d},r(a_1,\mathbf{d}))$-equitable, 
there are at most $(d_j+\epsilon_j)m^j$ many $j$-tuples in each $(\epsilon_j,d_j,r(a_1,\mathbf{d}))$-regular $(j,j-1)$-polyads, where $m\leq  \lceil n/a_1\rceil\leq n/k$.
Since the number of $(\epsilon_j,d_j,r(a_1,\mathbf{d}))$-regular $(j,j-1)$-polyads is at most 
$L_k^{2^{k-1}-2}/j!$, and $L_k\geq a_1\geq k>j\geq 2$ and $d_j\leq L^{-2^{k-1}}$, the number of $j$-tuples in $(\epsilon_j,d_j,r(a_1,\mathbf{d}))$-regular $(j,j-1)$-polyads is at most 
$$(d_j+\epsilon_j)m^j\cdot L_k^{2^{k-1}-2}/j!\leq d_j m^j\cdot L_k^{2^{k-1}-2}
\leq L_k^{-2}{n\choose j}\leq \frac{1}{9}{n\choose j},$$
the first inequality holds because, according to the definition of the function $\epsilon_i$,  $\epsilon_i\leq d_i$ for $2\leq i\leq k-1$.
This implies that the number of bad $j$-tuples is at least 
$\frac{8}{9}{n\choose j}$.
However, since each $j$-tuple is contained in ${n-j\choose k-j}$ many $k$-tuples, and there are at most 
$\mu {n\choose k}$ many $k$-tuples does not belong to any 
$(\boldsymbol{\epsilon}(\mathbf{d}),\mathbf{d},r(a_1,\mathbf{d}))$-regular $(k,k-1)$-complexes (by the first statement), it follows that the number of bad $j$-tuples is at most 
$$\frac{\mu{n\choose k}{k\choose j}}{{n-j\choose k-j}}=\mu{n\choose j}<\frac{8}{9}{n\choose j},$$
a contradiction.
\end{proof}

Additionally, we will also need the Supersaturation Lemma \cite{Erdos-Sim}, which is stated as follows.

\begin{theorem} [Supersaturation Lemma]\label{hypergraph-supersaturation} 
Let $\mathcal{F}$ be a $k$-uniform hypergraph with $k\geq 2$. For any $\epsilon\in(0,1)$, there
exist positive constants $\delta=\delta(\mathcal{F},\epsilon)\in(0,1)$ and $n_0=n_0(\mathcal{F},\epsilon)$ such that for any $n$-vertex $k$-uniform hypergraph $\mathcal{G}$ with $n\geq n_0$, if $\mathcal{G}$ has at least $ex(n,\mathcal{F})+\epsilon n^k$ edges, then it contains at least $\delta n^{v(\mathcal{F})}$ copies of $\mathcal{F}$.
\end{theorem}

\begin{lemma}\label{clm-thm1-1}
Let $\mathcal{F},\mathcal{G}$ be two $k$-uniform hypergraphs, and let $\mathcal{H}$ be a  $k$-uniform hypergraph consisting of $\mathcal{F}$-copies. Then $\mathcal{H}$ contains at most $O(n^{v(\mathcal{G})-1})$ non-rainbow copies of $\mathcal{G}$.
\end{lemma}
\begin{proof}
Since each non-rainbow copy of $\mathcal{G}$ in $\mathcal{H}$ contains two edges $e_1,e_2$ of the same color ($e_1,e_2$ come from the same  $\mathcal{F}$-copy), there are at most 
$$\frac{e(\mathcal{H})}{e(\mathcal{F})}\cdot {e(\mathcal{F})\choose 2}\leq \frac{n^k}{e(\mathcal{F})}\cdot {e(\mathcal{F})\choose 2}$$ 
choices of $e_1,e_2$ and there are at most $Cn^{v(\mathcal{G})-|e_1\cup e_2|}\geq Cn^{v(\mathcal{G})-k-1}$ copies of $\mathcal{G}$ containing $e_1$ and $e_2$ for some constant $C=C(\mathcal{G})>0$.
It follows that $\mathcal{H}$ contains at most 
$$\frac{n^k}{e(\mathcal{F})}\cdot {e(\mathcal{F})\choose 2}\cdot Cn^{v(\mathcal{G})-k-1}\leq  e(\mathcal{F})\cdot Cn^{v(\mathcal{G})-1}$$
copies of non-rainbow $\mathcal{G}$.
\end{proof}

Partial proofs in this paper also depend on packing and fractional packing in hypergraphs. 
For two $k$-uniform hypergraphs $\mathcal{G}$ and $\mathcal{H}$, let ${\mathcal{G}\choose \mathcal{H}}$ denote the set of all copies of $\mathcal{H}$ in $\mathcal{G}$.
A map $\varphi^*: {\mathcal{G}\choose \mathcal{H}}\rightarrow [0,1]$ such that for each $e\in E(\mathcal{G})$
$$\sum_{e\in E(\mathcal{H}'), \mathcal{H}'\in {\mathcal{G}\choose \mathcal{H}}}\varphi^*(\mathcal{H}')\leq 1$$
is called a {\em fractional $\mathcal{H}$-packing} of $\mathcal{G}$.
The maximum weight of $\sum_{\mathcal{H}\in {\mathcal{G}\choose \mathcal{H}}}\varphi^*(\mathcal{H})$ 
is denoted by $\nu^*_\mathcal{H}(\mathcal{G})$.
A map $\varphi: {\mathcal{G}\choose \mathcal{H}}\rightarrow \{0,1\}$ such that for each $e\in E(\mathcal{G})$
$$\sum_{e\in E(\mathcal{H}'), \mathcal{H}'\in {\mathcal{G}\choose \mathcal{H}}}\varphi(\mathcal{H}')\leq 1$$
is called an {\em $\mathcal{H}$-packing} of $\mathcal{G}$.
It is clear that the maximum weight of
$\sum_{\mathcal{H}\in {\mathcal{G}\choose \mathcal{H}}}\varphi(\mathcal{H})$ 
is the $\mathcal{H}$-packing number $\nu_\mathcal{H}(\mathcal{G})$ of  $\mathcal{G}$, which is defined in Section 1.
R\"{o}dl, Schacht, Siggers and Tokushige \cite{RSST} presented the following result, which extended a result by Haxell and R\"{o}dl \cite{fractional-graph}.

\begin{theorem}[\cite{RSST}]\label{thm-frac}
For every $k$-uniform hypergraph $\mathcal{F}$, and for all $\eta>0$, there
exists $N\in \mathbb{N}$, such that for all $n>N$ and all $k$-uniform hypergraphs $\mathcal{H}$ on $n$ vertices, $\nu^*_\mathcal{F}(\mathcal{H})-\nu_\mathcal{F}(\mathcal{H})<\eta n^k$.
\end{theorem}
For two graphs $\mathcal{G},\mathcal{H}$ and an edge $e\in E(\mathcal{H})$, we use $\mathcal{N}_e(\mathcal{G},\mathcal{H})$ to denote the  number of copies of $\mathcal{G}$ (not necessarily edge-disjoint) in $\mathcal{H}$ that contain the edge $e$.
\begin{lemma}\label{lem-two-packings}
    Let $\mathcal{F}$ be a given $k$-uniform hypergraph. For any $k$-uniform hypergraph $\mathcal{H}$ with $v(\mathcal{H})$ sufficiently large, if $\mathcal{N}_e(\mathcal{F},\mathcal{H})=\mathcal{N}_{e'}(\mathcal{F},\mathcal{H})=m$ for any two edges $e,e'\in E(\mathcal{H})$, then 
    $\nu_\mathcal{F}(\mathcal{H})=\frac{e(\mathcal{H})}{e(\mathcal{F})}+o(n^k)$.
\end{lemma}
\begin{proof}
    Define a map $\varphi^*: {\mathcal{F}\choose \mathcal{H}}\rightarrow [0,1]$ such that 
    $\varphi^*(\mathcal{F}')=1/m$ for each $\mathcal{F}'\in {\mathcal{H}\choose \mathcal{F}}$.
    It is clear that $\varphi^*$ is a fractional $\mathcal{F}$-packing of $\mathcal{H}$ and $$e(\mathcal{H})=\sum_{e\in E(\mathcal{H})}\sum_{e\in \mathcal{F}',\mathcal{F}'\in{\mathcal{H}\choose \mathcal{F}}}\varphi^*(\mathcal{F}')\leq e(\mathcal{F})\cdot\nu^*_\mathcal{F}(\mathcal{H}).$$
    By Theorem \ref{thm-frac}, we have that $\nu_\mathcal{F}(\mathcal{H})=\nu^*_\mathcal{F}(\mathcal{H})+o(n^k)\geq \frac{e(\mathcal{H})}{e(\mathcal{F})}+o(n^k)$.
\end{proof}

\section{Proof of Theorem \ref{main-1}}\label{section-thm-1}

In this section, we prove our main result Theorem \ref{main-1}. The following result implies that the sufficiency of Theorem \ref{main-1} holds.

\begin{lemma}\label{dense}
If there is no homomorphism from $\mathcal{G}$ to $\mathcal{F}$, then $\frac{n^k}{v(F)^k}+o(n^{k-1})\leq ex_\mathcal{F}(n,\mathcal{G})=\Theta(n^k)$.
\end{lemma}
\begin{proof}
Assume that $n$ is sufficiently large.
Since there is no homomorphism from $\mathcal{G}$ to $\mathcal{F}$, it follows that there is no homomorphism from $\mathcal{G}$ to $\mathcal{F}(s)$ for any integer $s>0$.
Suppose that $\lfloor n/v(\mathcal{F})\rfloor=t$ and $\mathcal{H}=\mathcal{F}(t)$ is a $t$-blowup of  
$\mathcal{F}$ on $n_1=t\cdot v(\mathcal{F})$ vertices.
Then $\mathcal{H}$ is $\mathcal{G}$-free.
It is clear that each edge of $\mathcal{H}$ is contained in the same number of copies of $\mathcal{F}$ spanned by one vertex from each part of $V(\mathcal{H})$.
By Lemma \ref{lem-two-packings}, 
we have $\nu_\mathcal{F}(\mathcal{H})\geq  e(\mathcal{H})/e(\mathcal{F})+o(n^k)=\frac{n^k}{v(F)^k}+o(n^{k-1}).$
On the other hand, it is obvious that 
$$ex_\mathcal{F}(n,\mathcal{G})\leq {n\choose k}/e(\mathcal{F})\leq \frac{n^k}{k!e(\mathcal{F})}.$$
Hence,  $ex_\mathcal{F}(n,\mathcal{G})= \Theta(n^k)$.
\end{proof}

Next, we prove the necessity of Theorem \ref{main-1}, that is, to prove that if there is a homomorphism from $\mathcal{G}$ to $\mathcal{F}$, then $ex_\mathcal{F}(n,\mathcal{G})=o(n^k)$.
Suppose to the contrary that there is a $c^*\in(0,1)$ such that for any positive integer $N$, there exists $n>N$ with $ex_{\mathcal{F}}(n,\mathcal{G})\geq c^*n^k$.
For sufficiently large  integer $n$, suppose that $\mathcal{H}^*$ is a $k$-uniform hypergraph consisting of $c^*n^k$ edge-disjoint copies of $k$-uniform hypergraph $\mathcal{F}$ ($\mathcal{F}$-copies), and that it contains no rainbow  $\mathcal{G}$.

Since there is a homomorphism from $\mathcal{G}$ to $\mathcal{F}$, there exists an integer $\ell$ such that 
\begin{align}\label{frac-1-6}
\left(1-\frac{v(\mathcal{F})}{\ell}\right)^{v(\mathcal{F})}\geq \frac{2}{3}
\end{align} and $\mathcal{G}$ is a subgraph of $\mathcal{F}(\ell)$ (here, $\mathcal{F}(\ell)$ denotes an $\ell$-blowup of $\mathcal{F}$). In following lemma, we first select a sub-hypergraph $\mathcal{H}^*$ from $\mathcal{H}$ with most edge-disjoint $\mathcal{F}$-copies,  which are crossing with respect to some equitable vertex partition.

\begin{lemma}\label{lem-crossing}
There is a equitable vertex partition $W_1,W_2,\ldots, W_{\ell}$ of $\mathcal{H}^*$ such that there are at least 
$$\frac{c^*n^k}{3v(\mathcal{F})!\cdot \ell^{\ell-v(\mathcal{F})}}$$ 
crossing $\mathcal{F}$-copies with respect to the equitable vertex partition. 
\end{lemma}
\begin{proof}
Delete $\mathcal{F}$-copies incident with $s\leq \ell-1$ vertices of $\mathcal{H}^*$ arbitrary  such that the resulting graph $\widetilde{\mathcal{H}}$ satisfies $\ell | v(\widetilde{\mathcal{H}})$.
It is clear that $\widetilde{\mathcal{H}}$ contains at least  $c^*n^k/2$ many $\mathcal{F}$-copies.

For positive integers $a,b$ with $a|b$, define 
$$\mu(b,a)=\frac{b!}{a ! \cdot \left[(b/a)! \right]^a}.$$
It is clear that there are $\mu(n-s,\ell)$ choices of the equitable vertex partition $U_1,U_2,\ldots, U_{\ell}$ of $V(\widetilde{\mathcal{H}})$.
For each $\mathcal{F}$-copy $\mathcal{F}^*$ in $\widetilde{\mathcal{H}}$, there are at least 
$${\ell\choose v(\mathcal{F})}\cdot \mu(n-\ell-s,\ell)$$
choices of the equitable vertex partition $U_1,U_2,\ldots, U_{\ell}$ such that $\mathcal{F}^*$ is crossing.
Hence, there exists a equitable vertex partition of $V(\widetilde{\mathcal{H}})$ such that there are at least 
\begin{align*}
\frac{{\ell\choose v(\mathcal{F})}\cdot \mu(n-\ell-s,\ell)\cdot c^*n^k/2}{\mu(n-s,\ell)}
&=\frac{\ell !}{v(\mathcal{F})!(\ell-v(\mathcal{F}))!}\cdot\frac{(n-\ell-s)!}{(n-s)!}\cdot
\left[\frac{((n-s)/\ell)!}{((n-\ell-s)/\ell)!}\right]^\ell \cdot \frac{c^*n^k}{2}\\
&\geq \frac{(\ell-v(\mathcal{F}))^{v(\mathcal{F})}}{v(\mathcal{F})!}\cdot \frac{1}{(n-s)^\ell} \cdot \frac{(n-s)^\ell}{\ell^\ell}\cdot \frac{c^*n^k}{2}\\
&=\frac{c^*[\ell-v(\mathcal{F})]^{v(\mathcal{F})}n^k}{2v(\mathcal{F})!\cdot \ell^{\ell}}
\geq \frac{c^*n^k}{3v(\mathcal{F})!\cdot \ell^{\ell-v(\mathcal{F})}} 
\end{align*}
$\mathcal{F}$-copies of $\widetilde{\mathcal{H}}$ are crossing with respect to the vertex partition (the last inequality is obtained by Ineq.~(\ref{frac-1-6})).  
Hence, there is a equitable vertex partition of $\mathcal{H}^*$, say $W_1,W_2,\ldots, W_{\ell}$, such that there are at least $\frac{c^*n^k}{3v(\mathcal{F})!\cdot \ell^{\ell-v(\mathcal{F})}}$ crossing $\mathcal{F}$-copies with respect to the equitable vertex partition. 
\end{proof}

By Lemma \ref{lem-crossing}, there is a sub-hypergraph $\mathcal{H}$ of $\mathcal{H}^*$ and a equitable vertex partition $W_1,W_2,\ldots, W_{\ell}$ of $\mathcal{H}$ such that each $\mathcal{F}$-copy is crossing with respect to the equitable vertex partition and $\mathcal{H}$ consists of at least $cn^k$ edges-disjoint $\mathcal{F}$-copies, where
\begin{align}\label{def-c}
c=\frac{c^*}{3v(\mathcal{F})!\cdot \ell^{\ell-v(\mathcal{F})}}.
\end{align}
Moreover, $\mathcal{H}$ is rainbow $\mathcal{G}$-free.
By Lemma \ref{clm-thm1-1}, in order to complete the proof, we only need to show that the following statement holds:
\begin{itemize}
  \item [$(\star)$.] there exists a positive real number $c'$ such that $\mathcal{H}$ contains at least $c'n^{v(\mathcal{G})}$ copies of $G$.
\end{itemize}


For $\mathcal{F}$ and $\nu=1/2$, Theorem \ref{counting-lemma} decides non-negative functions 
\begin{align}\label{fun-1}
\delta_k(x_k),\delta_{k-1}(x_{k-1},x_k),\ldots,\delta_{2}(x_2,\ldots,x_k),r'(x_2,\ldots,x_k)\mbox{ and }m_0(x_2,\ldots,x_k).
\end{align}
Below, we will complete the proof in three steps.

\bigskip

\noindent{\bf Hypergraph Regularity Lemma}.
In this part, we get a sub-hypergraph $\mathcal{H}_1$ of $\mathcal{H}$ satisfying  the following lemma. 

Let 
$\epsilon_k=\min\{\delta_k(c/3),c/12\}$, where $c$ is defined in Eq. (\ref{def-c}).
Let
$$\epsilon_i(x_i,\ldots,x_{k-1})=\min\{\delta_{i}(x_i,\ldots,x_{k-1},c/3),x_i\},$$ 
$2\leq i\leq k-1$, and $r(x_1,\ldots,x_{k-1})=r'(x_2,\ldots,x_{k-1},c/3)$
be functions.
Applying Theorem \ref{regularity-lemma} with the  vertex partition initiated by $\{W_1,W_2,\ldots,W_\ell\}$, we will get the following result.

\begin{lemma}\label{refined-reg}
For real numbers $\mu=c/12, \epsilon_k$ and non-negative functions $r,\epsilon_2,\ldots,\epsilon_{k-1}$ defined as above, there exist integers $N_{Reg}$ and $L_{Reg}$ such that the following holds.
If above $\mathcal{H}$ satisfies $v(\mathcal{H})\geq N_{Reg}$, there is a sub-hypergraph $\mathcal{H}_1$ of $\mathcal{H}$ with $V(\mathcal{H})=V(\mathcal{H}_1)$, a family of partitions $\mathscr{P}=\mathscr{P}(k-1, \mathbf{a})$ on $V(\mathcal{H}_1)$ (say $\mathscr{P}^{(1)}=\{V_1,\ldots,V_{a_1}\}$ and $\mathbf{a}=\{a_1,\ldots,a_{k-1}\}$) and a vector $\mathbf{d}=(d_2,\ldots,d_{k-1})$ such that
\begin{enumerate}
  \item [(I)] $\mathscr{P}$ is $(\mu, \boldsymbol{\epsilon}(\mathbf{d}),\mathbf{d},r(a_1,\mathbf{d}))$-equitable and $L_{Reg}$-bounded, where \begin{align*}
      \boldsymbol{\epsilon}(\mathbf{d})&=
      (\epsilon_{2}(d_2,\ldots,d_{k-1}),\ldots,\epsilon_{k-2}(d_{k-2},d_{k-1}),
      \epsilon_{k-1}(d_{k-1}))
      \end{align*}
      and $r(a_1,\mathbf{d})=r(a_1,d_2\ldots,d_{k-1})=r'(d_2,\ldots,d_{k-1},c/3)$;
  \item [(II)] For each $K\in\mathcal{H}_1$, $\mathcal{H}\cap \mathcal{K}_k(\widehat{\mathcal{P}}^{(k-1)}(K))$ is  $(\epsilon_k,d_K,r(a_1,\mathbf{d}))$-regular with respect to $\widehat{\mathcal{P}}^{(k-1)}(K)$ for some $d_K\geq c/3$;
  \item[(III)] every edge of $\mathcal{H}_1$ belongs to  some $(\boldsymbol{\epsilon}(\mathbf{d}),\mathbf{d},r(a_1,\mathbf{d}))$-regular $(k,k-1)$-complex;
  \item[(IV)] $\mathcal{H}_1$ contains at least $cn^k/2$ many $\mathcal{F}$-copies;
  \item [(V)] each $\mathcal{F}$-copy in $\mathcal{H}_1$ is crossing with respect to $\mathscr{P}^{(1)}$;
  \item [(VI)] $L_{Reg}^{-2^{k-1}}<d_j$ for each $2\leq j\leq k-1$.
\end{enumerate}
\end{lemma}
\begin{proof}
We apply Theorem \ref{regularity-lemma} with real numbers $\mu, \epsilon_k$ and non-negative functions $r,\epsilon_2,\ldots,\epsilon_{k-1}$ by assuming that the initial partition $\mathscr{P}^{(1)}$ refines $\{W_1,W_2,\ldots,W_\ell\}$.
Then by Theorem \ref{regularity-lemma}, there are two integers $N_{Reg}$ and $L_{Reg}$. For the hypergraph $\mathcal{H}$ with $v(\mathcal{H})=n\geq N_{Reg}$, there is a family of partitions $\mathscr{P}=\mathscr{P}(k-1, \mathbf{a})$ on $V(\mathcal{H}_1)$ and a vector $\mathbf{d}=(d_2,\ldots,d_{k-1})$
such that $\mathcal{H}$ is $(\epsilon_k,r(a_1,\mathbf{d}))$-regular with respect to $\mathscr{P}$  and the statement $(I)$ holds.
Moreover,  since each $\mathcal{F}$-copy in $\mathcal{H}$ is crossing with respect to the partition $\{W_1,W_2,\ldots,W_\ell\}$ and $\mathscr{P}^{(1)}$ refines $\{W_1,W_2,\ldots,W_\ell\}$, $\mathcal{F}$-copies in any sub-hypergraph of $\mathcal{H}$  is crossing with respect to $\mathscr{P}^{(1)}$, the statement $(V)$ holds.

Let $\mathcal{H}_1$ be the sub-hypergraph of $\mathcal{H}$ obtained by deleting each edge $K$ satisfying one of the following statement (for all related notions, please refer to the Remark \ref{relark-com} and Definitions \ref{definition-uedr} and \ref{definition-br}).
\begin{itemize}
\item [(A).] $\boldsymbol{\widehat{\mathcal{P}}}^{(k-1)}(K)=\{\widehat{\mathcal{P}}^{(i)}(K)\}_{i=1}^{k-1}$ is not an $(\epsilon(\mathbf{d}),\mathbf{d},r(a_1,\mathbf{d}))$-regular $(k,k-1)$-complex.
  \item [(B).] $\mathcal{H}\cap \mathcal{K}_k(\widehat{\mathcal{P}}^{(k-1)}(K))$ is  $(\epsilon_k,d_K,r(a_1,\mathbf{d}))$-regular with respect to $\widehat{\mathcal{P}}^{(k-1)}(K)$ for some $d_K\in[0,c/3)$. 
  \item [(C).] $\mathcal{H}\cap \mathcal{K}_k(\widehat{\mathcal{P}}^{(k-1)}(K))$ is not $(\epsilon_k,r(a_1,\mathbf{d}))$-regular with respect to $\widehat{\mathcal{P}}^{(k-1)}(K)$. 
\end{itemize}
It is clear that there are at most $cn^k/3$ edges of $\mathcal{H}$ satisfying (B).
By statement $(I)$ in Lemma \ref{refined-reg}, there are at most $\mu {n\choose k}$ edges of $\mathcal{H}$ satisfy (A). In addition, there are at most $\epsilon_k {n\choose k}$ edges of $\mathcal{H}$ satisfy (C).
Therefore, we  remove at most $\mu {n\choose k}+cn^k/3+\epsilon_k {n\choose k}\leq cn^k/2$ edges from $\mathcal{H}$ in total. Consequently, the resulting sub-hypergraph $\mathcal{H}_1$ contains at least $cn^k/2$  $\mathcal{F}$-copies, which implies that the statement $(IV)$ hold.
Additionally, the edge-deleting methods ensure that the statements $(II)$ and $(III)$ hold.

By the definition of $\epsilon_i$, we have that $\epsilon_i\leq d_i$.
Note that $\mu=c/12<8/9$. By Lemma \ref{relation-d-L},  the statements $(II)$ and $(III)$ holds.
\end{proof}

\bigskip

\noindent{\bf Counting Lemma}.
In this part, we count the number of copies of $\mathcal{F}$ in $\mathcal{H}_1$.
For real numbers $\mu=c/12, \epsilon_k$ and non-negative functions $r,\epsilon_2,\ldots,\epsilon_{k-1}$ defined as above, by Lemma \ref{refined-reg}, there exist integers $N_{Reg}$ and $L_{Reg}$.
Let
$$N=\max\{L_{Reg}\cdot m_0(L_{Reg}^{-2^{k-1}},\ldots,L_{Reg}^{-2^{k-1}},c/3), N_{Reg}\}.$$
We choose the hypergraph $\mathcal{H}$ satisfying $v(\mathcal{H})=n\gg N$.
Above Lemma \ref{refined-reg} indicates that there is a sub-hypergraph $\mathcal{H}_1$ of $\mathcal{H}$ with $V(\mathcal{H})=V(\mathcal{H}_1)$, a family of partitions $\mathscr{P}=\mathscr{P}(k-1, \mathbf{a})$ on $V(\mathcal{H}_1)$ (say $\mathscr{P}^{(1)}=\{V_1,\ldots,V_{a_1}\}$ and $\mathbf{a}=\{a_1,\ldots,a_{k-1}\}$) and a vector $\mathbf{d}=(d_2,\ldots,d_{k-1})$ such that statements $(I)$--$(VI)$ hold.
By the statements $(IV)$ and $(V)$ of Lemma \ref{refined-reg}, $\mathcal{H}_1$ contains a copy $\mathcal{F}_0$ of $\mathcal{F}$ and $\mathcal{F}_0$ is crossing with respect to the vertex partition $\mathscr{P}^{(1)}$ in Lemma \ref{refined-reg}.
Without loss of generality, suppose that $V(\mathcal{F}_0)=\{v_1,\ldots,v_f\}$ and $v_i\in V_i$ for each $i\in[f]$. Since $\mathscr{P}^{(1)}=\{V_1,\ldots,V_{a_1}\}$ and $a_1\leq L_{Reg}$ by the statement $(I)$ in Lemma \ref{refined-reg}, it follows that
$$ n/L_{Reg}\leq \lfloor n/a_1\rfloor\leq|V_1|\leq\ldots\leq|V_f|\leq \lceil n/a_1\rceil.$$ 
Set $m=|V_1|$. Then $m\geq m_0(L_{Reg}^{-2^{k-1}},\ldots,L_{Reg}^{-2^{k-1}},c/3)\geq m_0(\mathbf{d},c/3)$ since $m_0$ is non-increasing in every variable and $d_i>L_{Reg}^{-2^k}$ for each $2\leq i\leq k-1$.

For $d_k=c/3$ and the vector $\mathbf{d}=(d_2,\ldots,d_{k-1})$.
Note that $\mathcal{H}_1[\bigcup_{i\in[f]}V_i]$ determines a complex $\boldsymbol{\mathcal{H}}_{\mathcal{F}_0}=\{\mathcal{H}^{(i)}\}_{i=1}^k$, where $\mathcal{H}^{(k)}=\mathcal{H}_1[\bigcup_{i\in[f]}V_i]$.
For each edge $K\in \mathcal{H}^{(k)}$, by statement $(III)$ of Lemma \ref{refined-reg}, we have that 
$\{\mathcal{H}^{(i)}[\bigcup_{j\in K}V_j]\}_{i=1}^{k-1}$ is an $(\boldsymbol{\epsilon}(\mathbf{d}),\mathbf{d},r(a_1,\mathbf{d}))$-regular complex, 
and also is a $(\boldsymbol{\delta}(\mathbf{d}),\mathbf{d},r(a_1,\mathbf{d}))$-regular complex since $\epsilon_i(d_2,\ldots,d_{k-1})\leq \delta_i(d_2,\ldots,d_{k-1},c/3)$ for each $2\leq i\leq k-1$.
By the statement $(II)$ of Lemma \ref{refined-reg}, we have 
that $\mathcal{H}\cap\mathcal{K}_k(\widehat{\mathcal{P}}^{(k-1)}(K))$ is 
$(\epsilon_k, d',r(a_1,\mathbf{d}))$-regular with respect to $\widehat{\mathcal{P}}^{(k-1)}(K)$ for some $d'\geq c/3=d_k$.
Hence, $\boldsymbol{\mathcal{H}}_{\mathcal{F}_0}$ is an $((\boldsymbol{\delta}(\mathbf{d}),\epsilon_k),(\mathbf{d},d_k),r(a_1,\mathbf{d}))$-regular $(m,\mathcal{F}_0)$-complex with $m\geq m_0(\mathbf{d},d_k)$ (note that $m\geq n/L_{Reg}$ also holds).
Hence, by Theorem \ref{counting-lemma}, there are  at least $\gamma n^f$ copies of $\mathcal{F}$ in $\mathcal{H}_1$, where 
$$\gamma=(1-\nu)L_{Reg}^{-f}\cdot d_k^{|\mathcal{F}|}\cdot\prod_{i=2}^{k-1}\cdot d_i^{|\Delta_i(\mathcal{F})|}\geq\frac{(c/3)^{|\mathcal{F}|}}{2} \cdot L_{Reg}^{-f-2^{k-1}\cdot\sum_{i=2}^{k-1}|\Delta_i(\mathcal{F})|}$$
is a positive real number.

\bigskip

\noindent{\bf Supersaturation lemma}.
In this part, we count the number of copies of $\mathcal{G}$ in $\mathcal{H}$.
We first introduce a notation $p_{\mathcal{F}}$ which represents the maximum number of edge-disjoint copies of $\mathcal{F}$ that can be embedded in $K_{v(\mathcal{F})}^{(k)}$.
For each copy $\mathcal{F}_1$ of $\mathcal{F}$ in $\mathcal{H}$, 
there is a permutation $\Phi_{\mathcal{F}_1}:[f]\rightarrow  [f]$ such that each $K\in V(\mathcal{F}_0)$ is uniquely associated with an edge $K'\in\mathcal{F}'$ that is crossing with respect to $\{v_i\in K:V_{\Phi_{\mathcal{F}_1}(i)}\}$.
It is clear that there is a permutation $\Phi:[f]\rightarrow  [f]$ such that there are at least $\frac{\gamma n^f}{p_{\mathcal{F}}\cdot f!}$ copies $\mathcal{F}_1$ of $\mathcal{F}$ with $ \Phi_{\mathcal{F}_1}=\Phi$ and any two such copies do not share the same vertex-set.
Let $\mathcal{H}'$ be the sub-hypergraph consisting of these copies of $\mathcal{F}$ and let $\mathcal{R}$ be an $f$-uniform hypergraph with $V(\mathcal{R})=V(\mathcal{H}')$ and $E(\mathcal{R})=\{V(\mathcal{F}_1):\mathcal{F}_1\mbox{ is a copy of }\mathcal{F}\mbox{ in }\mathcal{H}'\}$.
Since $e(\mathcal{R}) \geq \frac{\gamma n^f}{p_{\mathcal{F}}\cdot f!}$  and $\pi_f\left(K_f^{(f)}(v(\mathcal{G}))\right)=0$ (where $K_f^{(f)}(v(\mathcal{G}))$ represents a complete $f$-part $f$-uniform hypergraph with parts of equal size  $v(\mathcal{G})$), when $n$ is sufficiently large, there exists a $\widetilde{c}>0$ such that $\mathcal{R}$ contains at least $\widetilde{c}\cdot n^{f\cdot v(\mathcal{G})}$ copies of $K_f^{(f)}(v(\mathcal{G}))$ by Theorem \ref{hypergraph-supersaturation}.
Hence, $\mathcal{H}'$ contains at least $\widetilde{c}\cdot n^{f\cdot v(\mathcal{G})}$ copies of $\mathcal{F}(v(\mathcal{G}))$.
Since there is a homomorphism from $\mathcal{G}$ to $\mathcal{F}$, it follows that each $\mathcal{F}(v(\mathcal{G}))$ contains a copy of $\mathcal{G}$.
Since  each copy of $\mathcal{G}$ in $\mathcal{H}$ is contained in at most 
$Cn^{f\cdot v(\mathcal{G})-v(\mathcal{G})}$ copies of  $\mathcal{F}(v(\mathcal{G}))$ for some constant $C=C(\mathcal{F},\mathcal{G})>0$, it follows that $\mathcal{H}$ contains at least $c' n^{v(\mathcal{G})}$ copies of $\mathcal{G}$, where $c'=\frac{\widetilde{c}}{C}$.
Therefore, the statement ($\star$) holds. The proof is completed.
\hfill$\square$ 

\section{Proofs of Proposition \ref{super-sat} and Theorem \ref{main-2}}\label{section-thm-2}

\noindent{\bf Proof of Proposition \ref{super-sat}}:
Choose $\xi>0$ arbitrary.
Let $\mathcal{S}$ be the set of $k$-uniform hypergraphs $\mathcal{G}'$ with $v(\mathcal{G}')\leq v(\mathcal{G})$ such that there exists a homomorphism from  $\mathcal{G}$ to $\mathcal{G}'$. It is clear that $\mathcal{S}$ is a finite set. Without loss of generality, assume that $\mathcal{S}=\{\mathcal{G}_i:i\in[p]\}$.

By Theorem \ref{counting-lemma}, for a real number $\nu\in(0,1)$ and each $\mathcal{G}_j\in \mathcal{S}$,
there exists non-negative functions 
$$\delta_{k,j}(x_k),\delta_{k-1,j}(x_{k-1},x_k),\ldots,\delta_{2,j}(x_2,\ldots,x_k),r'_j(x_2,\ldots,x_k)$$ and 
$m_{0,j}(x_2,\ldots,x_k)$.
Let
$$\delta_i(x_i,\ldots,x_k)=\min\{\delta_{i,j}(x_i,\ldots,x_k):j\in[p]\}$$ 
for each $2\leq i\leq k$, $r'(x_2,\ldots,x_k)=\max\{r'_j(x_2,\ldots,x_k):j\in[p]\}$ and  
$$m_0(x_2,\ldots,x_k)=\max\{m_{0,j}(x_2,\ldots,x_k):j\in[p]\}.$$
For $\epsilon_k=\min\{\delta_k(\xi/10),\xi/10\}$, 
$\mu=\xi/10$, 
$$r(x_1,\ldots,x_{k-1})=r'(x_2,\ldots,x_{k-1},\xi/10)$$ 
and 
$$\epsilon_i(x_i,\ldots,x_{k-1})=\delta_i(x_i,\ldots,x_{k-1},\xi/10)$$ 
for $2\leq i\leq k-1$, Theorem \ref{regularity-lemma} determines integers $N_{Reg}$ and $L_{Reg}$.
Let
$$N_0=\max\left\{N_{Reg},L_{Reg}\cdot m_0(L_{Reg}^{-2^{k-1}},\cdots,L_{Reg}^{-2^{k-1}},\xi/10)\right\}.$$
Choose a $k$-uniform hypergraph $\mathcal{H}$ arbitrary such that $v(\mathcal{H})=n\gg N_0$ and $\mathcal{H}$ consists of $ex_\mathcal{F}(n,\mathcal{G})+\xi n^k$ $\mathcal{F}$-copies.
Apply Theorem \ref{regularity-lemma} again, 
there exists a family of partitions $\mathscr{P}=\mathscr{P}(k-1, \mathbf{a})$ on $V(\mathcal{H})$ (say $\mathbf{a}=(a_1,\ldots,a_{k-1})$) and a vector $\mathbf{d}=(d_2,\ldots,d_{k-1})\in(0,1]^{k-2}$ so that
\begin{enumerate}
  \item $\mathscr{P}$ is $(\mu, \boldsymbol{\epsilon}(\mathbf{d}),\mathbf{d},r(a_1,\mathbf{d}))$-equitable and $L_{Reg}$-bounded, where $$\boldsymbol{\epsilon}(\mathbf{d})=(\epsilon_{2}(d_2,\ldots,d_{k-1}),\ldots,\epsilon_{k-2}(d_{k-2},d_{k-1}),
      \epsilon_{k-1}(d_{k-1}))$$
      and $r(a_1,\mathbf{d})=r(a_1,d_2,\ldots,d_{k-1})$;
  \item $\mathcal{H}$ is $(\delta_k,r(a_1,\mathbf{d}))$-regular with respect to $\mathscr{P}$.
\end{enumerate}

Let $\mathcal{H}_1$ be the sub-hypergraph of $\mathcal{H}$ obtained by deleting each edge $K$ satisfying one of the following statement (for all related notions, please refer to the Remark \ref{relark-com} and Definitions \ref{definition-uedr} and \ref{definition-br}).
\begin{itemize}
\item [(A).] $\boldsymbol{\widehat{\mathcal{P}}}^{(k-1)}(K)=\{\widehat{\mathcal{P}}^{(i)}(K)\}_{i=1}^{k-1}$ is not an $(\epsilon(\mathbf{d}),\mathbf{d},r(a_1,\mathbf{d}))$-regular $(k,k-1)$-complex.
\item [(B).] $\mathcal{H}\cap \mathcal{K}_k(\widehat{\mathcal{P}}^{(k-1)}(K))$ is  $(\epsilon_k,d_K,r(a_1,\mathbf{d}))$-regular with respect to $\widehat{\mathcal{P}}^{(k-1)}(K)$ for some $d_K\in[0,\xi/10)$. 
\item [(C).] $\mathcal{H}\cap \mathcal{K}_k(\widehat{\mathcal{P}}^{(k-1)}(K))$ is not $(\epsilon_k,r(a_1,\mathbf{d}))$-regular with respect to $\widehat{\mathcal{P}}^{(k-1)}(K)$. 
\end{itemize}
As discussed in the proof of Lemma \ref{refined-reg}, we remove at most $\mu {n\choose k}+\xi n^k/10+\epsilon_k {n\choose k}< \xi n^k/3$ edges from $\mathcal{H}$ in total.
Since the edge removal destroys at most $\xi n^k/3$ $\mathcal{F}$-copies, there are at least $ex_\mathcal{F}(n,\mathcal{G})+2\xi n^k/3$ $\mathcal{F}$-copies, which implies $\mathcal{H}_1$ contains a copy of $\mathcal{G}$ (denote it by $\mathcal{G}_0$).

Let $U=\{i\in[a_i]:V_i\cap V(\mathcal{G}_0)\neq \emptyset\}$.
For each edge $e\in \mathcal{G}_0$, let $I_e=\{i\in[a_1]: V_i\cap e\neq \emptyset\}$. We construct a hypergraph  $\mathcal{G}'$ with $V(\mathcal{G}')=U$ and $E(\mathcal{G}')=\{I_e:e\in E(\mathcal{G})\}$.
Since each edge of $\mathcal{H}_1$ is crossing with respect to the partition $\mathscr{P}$, $\mathcal{G}'$ is a $k$-uniform hypergraph and $\mathcal{G}'\in\mathcal{S}$ (say $\mathcal{G}'$ is a copy of $\mathcal{G}_\ell$).
It is clear that $\mathcal{H}_1[\bigcup_{i\in U}V_i]$ determines a complex $\boldsymbol{\mathcal{H}}_{\mathcal{G}'}=\{\mathcal{H}^{(i)}\}_{i=1}^k$, where $\mathcal{H}^{(k)}=\mathcal{H}_1[\bigcup_{i\in U}V_i]$.
Let $d_k=\xi/10$.
\begin{claim}\label{clm-2-2}
$\boldsymbol{\mathcal{H}}_{\mathcal{G}'}$ is a $(\boldsymbol{\delta}_\ell, (\mathbf{d},\geq d_k), r_\ell)$-regular  $(m,\mathcal{G}_\ell)$-complex with  
$$\boldsymbol{\delta}_\ell=(\delta_{2,\ell}(d_2,\ldots,d_k),\ldots,\delta_{k-1,\ell}(d_{k-1},d_k),\delta_{k,\ell}(d_k)),$$ 
$r_\ell=r'_\ell(\mathbf{d},d_k)$ and $m\geq m_\ell(d_2,\cdots,d_k)$. 
\end{claim}
\begin{proof}
Since $d_k=\xi/10$ and each edge of $\mathcal{H}_1$ does not satisfy (A), (B) or (C), it follows that $\boldsymbol{\mathcal{H}}_{\mathcal{G}'}$ is a $(\boldsymbol{\delta}, (\mathbf{d},\geq d_k), r')$-regular  $(m,\mathcal{G}')$-complex with  
$$\boldsymbol{\delta}=(\delta_{2}(d_2,\ldots,d_k),\ldots,\delta_{k-1}(d_{k-1},d_k),\delta_k(d_k)),$$ 
$r'=r'(\mathbf{d},d_k)$ and $m\gg N_0/L_{Reg}\geq m_0(L_{Reg}^{-2^{k-1}},\cdots,L_{Reg}^{-2^{k-1}},\xi/10)$.
Since $\delta_{i,\ell}(d_i,\ldots,d_k)\geq \delta_{i}(d_i,\ldots,d_k)$ for each $2\leq i\leq k$ and $r_\ell\leq r'$, it follows that  $\boldsymbol{\mathcal{H}}_{\mathcal{G}'}$ is also a $(\boldsymbol{\delta}_\ell, (\mathbf{d},\geq d_k), r_\ell)$-regular  $(m,\mathcal{G}_\ell)$-complex with  
$$\boldsymbol{\delta}_\ell=(\delta_{2,\ell}(d_2,\ldots,d_k),\ldots,\delta_{k-1,\ell}(d_{k-1},d_k),\delta_{k,\ell}(d_k)),$$ 
$r_\ell=r'_\ell(\mathbf{d},d_k)$ and $m\gg m_0(L_{Reg}^{-2^{k-1}},\cdots,L_{Reg}^{-2^{k-1}},\xi/10)\geq m_\ell(d_2,\cdots,d_k)$, the last inequality holds since the function $m_0(x_2,\ldots,x_k)$ is non-increasing in every variable by Lemma \ref{relation-d-L}, and since $d_i> L_{Reg}^{-2^{k-1}}$ by Lemma \ref{dense}.
\end{proof}

By Theorem \ref{counting-lemma} and Claim \ref{clm-2-2}, $\mathcal{H}[\bigcup_{i\in U}V_i]$ has at least 
$$(1-\nu)m^{|\mathcal{G}_\ell|}\cdot d_k^{|\mathcal{G}_\ell|}\cdot\prod_{i=2}^{k-1} d_i^{|\Delta_i(\mathcal{G}_\ell)|}\geq \eta_1 n^{|\mathcal{G}_\ell|}$$
(not necessarily rainbow) copies of $\mathcal{G}_\ell$, for some constant $\eta_1>0$.
Fix a copy $\mathcal{G}^*$ of $\mathcal{G}_\ell$ in $\mathcal{H}[\bigcup_{i\in U}V_i]$.
For each copy $\mathcal{G}'$ of $\mathcal{G}_\ell$, let $\Phi_{\mathcal{G}'}$ be an isomorphism from
$\mathcal{G}'$ to $\mathcal{G}^*$.
Note that there are at most $|U|!$ such isomorphisms.
Therefore, there are at least $\eta_1 n^{v(\mathcal{G}_\ell)}(p_{\mathcal{G}_\ell}\cdot|U|!)^{-1}$ copies $\mathcal{G}'$ of $\mathcal{G}$ in $\mathcal{H}[\bigcup_{i\in U}V_i]$ that share the same isomorphism $\Phi_{\mathcal{G}'}$ but do not share the same vertex-set, where $p_{\mathcal{G}_\ell}$ is defined as the maximum number of edge-disjoint copies of $\mathcal{G}_\ell$ that can be embedded in $K_{v(\mathcal{G}_\ell)}^{(k)}$. Denote this family of copies by $\mathcal{J}$.
Let $\mathcal{R}$ be a $v(\mathcal{G}_\ell)$-uniform hypergraph on $\bigcup_{i\in U}V_i$ whose edge set
is $\{V(\mathcal{G}'): \mathcal{G}'\in  \mathcal{J}\}$.
It is clear that $e(\mathcal{R})\geq \eta_1 n^{v(\mathcal{G}_\ell)}(p_{\mathcal{G}_\ell}\cdot|U|!)^{-1}$.

Since there exists a homomorphism from $\mathcal{G}$ to $\mathcal{G}_\ell$, there is a positive integer $t$  such that $\mathcal{G}_\ell(t)$ contains $\mathcal{G}$ as a subgraph.
Let $v(\mathcal{G}_\ell)=\alpha$.
Given that $\pi\left(K^{(\alpha)}_{\alpha}(t)\right)=0$ (where $K^{(\alpha)}_{\alpha}(t)$ represents a complete $\alpha$-part  $\alpha$-uniform hypergraph with parts of equal size $t$), Theorem \ref{hypergraph-supersaturation} implies the existence of a constant $\eta_2>0$ such that $\mathcal{R}$ contains $\eta_2 n^{t\alpha}$ copies of $K^{(\alpha)}_{\alpha}(t)$.
As each $\mathcal{G}'\in  \mathcal{J}$ has the same $\Phi_{\mathcal{G}'}$, $\mathcal{H}[\bigcup_{i\in U}V_i]$ contains $\eta_2 n^{t\alpha}$ copies of $\mathcal{G}_\ell(t)$.
There exists a constant $C>0$ such that each copy of $\mathcal{G}$ is contained in at most $Cn^{t\alpha-v(\mathcal{G})}$ copies of $\mathcal{G}_\ell(t)$.
Consequently, $\mathcal{H}[\bigcup_{i\in U}V_i]$ contains $\frac{\eta_2}{C} n^{v(\mathcal{G})}$ copies of $\mathcal{G}$.
Then $\mathcal{H}[\bigcup_{i\in U}V_i]$ contains $\eta_2 n^{v(\mathcal{G})}(C\cdot p_{\mathcal{G}})^{-1}=\eta_3 n^{v(\mathcal{G})}$ copies of $\mathcal{G}$ such that any two copies of $\mathcal{G}$ do not share common set of vertex.
Denote by $\mathcal{J}'$ the set of these copies of $\mathcal{G}$.

Let $\mathcal{R}'$ be a $v(\mathcal{G})$-uniform hypergraph of $\bigcup_{i\in U}V_i$ whose edge set
is $\{V(\mathcal{G}''): \mathcal{G}''\in\mathcal{J}'\}$. 
Following the reasoning above and using the auxiliary hypergraph $\mathcal{R}'$, we conclude that there exists some $\eta_4>0$ such that $\mathcal{H}[\bigcup_{i\in U}V_i]$   (and hence $\mathcal{H}$ itself) contains
 $\eta_4 n^{v(\mathcal{G}(s))}$ copies of $\mathcal{G}(s)$.
By Lemma \ref{dense}, $\mathcal{H}$ contains at least $\eta_4 n^{v(\mathcal{G}(s))}/2=\eta n^{v(\mathcal{G}(s))}$ rainbow copies of  $\mathcal{G}(s)$.
\hfill$\square$ 
\bigskip

\noindent{\bf Proof of Theorem \ref{main-2}:} It is clear that $ex_\mathcal{F}(n,\mathcal{G}(s))\geq ex_\mathcal{F}(n,\mathcal{G})$.
Suppose, to the contrary, that $ex_\mathcal{F}(n,\mathcal{G}(s))\geq ex_\mathcal{F}(n,\mathcal{G})+cn^k$ for some $c>0$. 
For sufficiently large $n$, there is a rainbow $\mathcal{G}(s)$-free $k$-uniform hypergraph on $n$ vertices consisting of $ex_\mathcal{F}(n,\mathcal{G})+cn^k$ $\mathcal{F}$-copies.
By Proposition \ref{dense}, this is impossible.
\hfill$\square$ 
\bigskip

\section{Proofs of other results}\label{section-thm-3}

\noindent{\bf Proof of Theorem \ref{thm-triangle}}:
First, we claim that the following results hold for any two $k$-uniform hypergraphs $\mathcal{F}'$ and $\mathcal{G}'$.
\begin{enumerate}
    \item For any $k$-uniform hypergraph  $\mathcal{G}''$ containing $\mathcal{G}'$ as a subgraph, $ex_{\mathcal{F}'}(n,\mathcal{G}'')\geq ex_{\mathcal{F}'}(n,\mathcal{G}')$.
    \item For any $k$-uniform hypergraph  $\mathcal{F}''$ containing $\mathcal{F}'$ as a subgraph, $ex_{\mathcal{F}'}(n,\mathcal{G}')\geq ex_{\mathcal{F}''}(n,\mathcal{G}')$.
\end{enumerate}
Since $\mathcal{G}$ contains a copy of $H_3^{(k)}$ and there is a homomorphism form $\mathcal{G}$ to $\mathcal{F}$,  it follows that $v(\mathcal{F})\geq k+1$ and 
$$ex_{\mathcal{F}}(n,\mathcal{G})\geq ex_{K_{v(\mathcal{F})}^{(k)}}(n,\mathcal{G})\geq ex_{K_{v(\mathcal{F})}^{(k)}}(n,H_3^{(k)}).$$
Therefore, in combination with Theorem \ref{main-1}, it suffices to show that $ex_{K_r^{(k)}}(n,H_3^{(k)})\geq n^{k-o(1)}$ when $r\geq k+1$. 

By Theorem \ref{thm-Alon-Shapira} and Remark \ref{remark-1}, there exists an $n$-vertex $r$-uniform hypergraph $\mathcal{H}$ with $e(\mathcal{H})=n^{k-o(1)}$  such that any two edges intersect in at most $k-1$ vertices and any $3r-2k+1$ vertices span at most two edges.
Define $\mathcal{H}^*$ as the $k$-uniform hypergraph where $V(\mathcal{H}^*)=V(\mathcal{H})$ and $E(\mathcal{H}^*)$ consists of all $k$-subsets belonging to some edge of $\mathcal{H}$.
It is clear that each edge in $\mathcal{H}$ corresponds to a complete $k$-uniform hypergraph $K_r^{(k)}$ in $\mathcal{H}^*$.
Since any two edges in $\mathcal{H}$ intersect in at most $k-1$ vertices, it follows that $\mathcal{H}^*$ consists of $n^{k-o(1)}$ edge-disjoint copies $K_r^{(k)}$.
To complete the proof, we only need to show that $\mathcal{H}^*$ contains no rainbow $H_3^{(k)}$. 
Otherwise, there are three $K_r^{(k)}$-copies in $\mathcal{H}^*$, say $A_1,A_2,A_3$, and a vertex subset $S\subseteq V(\mathcal{H}^*)$ of size $s=k+1$ such that each  $A_i$  intersects $S$ in one edge.
Since $A_i$ and $A_j$ are edge-disjoint for $i\neq j$, it follows that $V(A_i)-S$ and $V(A_j)-S$ are disjoint.
Hence, $R=V(A_1\cup A_2\cup A_3)$
is a $(3r-2k+1)$-subset of $V(\mathcal{H})$ and $\mathcal{H}[R]$ contains at least three edges, a contradiction.
\hfill$\square$

\bigskip

\noindent{\bf Proof of Theorem \ref{thm-r-part}}: 
The lower bound is provided by Lemma \ref{dense}. For the upper bound, suppose for contradiction that there exists $\epsilon>0$ such that 
$$ex_{\mathcal{F}}(n,e(\mathcal{G}))\geq e(\mathcal{F})^{-1}ex_k(n,\mathcal{G})+\epsilon n^k$$ 
for infinitely many $n$.
For each such integer $n$, there exists an $n$-vertex $k$-uniform hypergraph $\mathcal{H}_n$ that consists of at least
$e(\mathcal{F})^{-1}ex_k(n,\mathcal{G})+\epsilon n^k$ $\mathcal{F}$-copies, with no rainbow copy of $\mathcal{G}$.
Then $e(\mathcal{H}_n)\geq ex_k(n,\mathcal{G})+\epsilon n^k$.
By Theorem \ref{hypergraph-supersaturation}, there exists a $\beta>0$ such that $\mathcal{H}_n$ contains $\beta n^{v(\mathcal{G})}$ copies of $\mathcal{G}$.
However, by Lemma \ref{clm-thm1-1}, $\mathcal{H}_n$ must contain a rainbow copy of $\mathcal{G}$, a contradiction.

We now show that if $\chi(\mathcal{F})=k$, then  $ex_{\mathcal{F}}(N,\mathcal{G})= e(\mathcal{F})^{-1}ex_k(N,\mathcal{G})+o(n^k)$ holds for sufficiently large $N$.
It is suffices to show that $ex_{\mathcal{F}}(N,\mathcal{G})\geq e(\mathcal{F})^{-1}ex_k(N,\mathcal{G})+o(n^k)$. 
Choose two integers $s,n$ such that $n=\lfloor N/s \rfloor$ and $1\ll s,n\ll N$.
Since there is no homomorphism form $\mathcal{G}$ to $\mathcal{F}$, it follows that $\chi(\mathcal{G})>k$, and thus 
$\pi(\mathcal{G})=\lim_{p\rightarrow \infty}ex(p,\mathcal{G})/{p\choose k}>0$. Then we can choose an $n$-vertex $\mathcal{G}$-free $k$-uniform hypergraph 
$\mathcal{H}^{ex}_n$ with $$e(\mathcal{H}^{ex}_n)=\pi(\mathcal{G}){n\choose k}+o(n^k)=\frac{\pi(\mathcal{G})n^k}{k!}+o(n^k).$$
It is clear that the $s$-blowup $\mathcal{H}^{ex}_n(s)$ of $\mathcal{H}^{ex}_n$ is also $\mathcal{G}$-free (say each vertex $x\in V(\mathcal{H}^{ex}_n)$ is expanded to an $s$-set $S_x$). To see this, we suppose to the contrary that $\mathcal{G}'$ is a copy of $\mathcal{G}$ in  $\mathcal{H}^{ex}_n(s)$. Let $U=\{x\in V(\mathcal{H}^{ex}_n):S_x\cap V(\mathcal{G}')\neq \emptyset\}$.
It is clear that there is a homomorphism from $\mathcal{G}$ to $\mathcal{H}^{ex}_n[U]$.
However, since $n$ is sufficiently large, the Blowup Theorem implies $\mathcal{H}^{ex}_n$ contains a copy of $\mathcal{G}$, a contradiction.

Note that each edge $e$ of $\mathcal{H}^{ex}_n$ is replaced by a copy of the complete $r$-part $r$-uniform hypergraph with parts of equal size $s$ in $\mathcal{H}^{ex}_n(s)$ (denoted by $\mathcal{B}_e$), and for any two edges $f,g$ of $\mathcal{H}^{ex}_n$, $\mathcal{B}_f$ and $\mathcal{B}_g$ are edge-disjoint.
The $sn$-vertex $k$-uniform hypergraph $\mathcal{H}^{ex}_n(s)$ can be regarded as a graph  consisting of $e(\mathcal{H}^{ex}_n)$ pairwise edge-disjoint copies of the complete $r$-part $r$-uniform hypergraph with parts of equal size $s$.

For each $e\in \mathcal{H}^{ex}_n$, since $\mathcal{B}_e$  is a complete $r$-part $r$-uniform hypergraph with parts of equal size $s$, each edge in $\mathcal{B}_e$ is contained in the same number of copies of $\mathcal{F}$.
Hence, by Lemma \ref{lem-two-packings}, each $\mathcal{B}_e$ contains $\frac{s^k}{e(\mathcal{F})}+o((ks)^k)$ edge-disjoint copies of $\mathcal{F}$, and  hence $\mathcal{H}^{ex}_n(s)$ contains 
$$e(\mathcal{H}^{ex}_n)\cdot\left[\frac{s^k}{e(\mathcal{F})}+o((ks)^k)\right]=\frac{\pi(\mathcal{G})(sn)^k}{k!}+o((sn)^k)=\frac{\pi(\mathcal{G})N^k}{k!}+o(N^k)$$
edge-disjoint copies of $\mathcal{F}$.
Since $\mathcal{H}^{ex}_n(s)$ is $\mathcal{G}$-free, $\mathcal{H}^{ex}_n(s)$ does not contain rainbow $\mathcal{G}$.
Consequently, 
$$ex_{\mathcal{F}}(N,\mathcal{G})\geq ex_{\mathcal{F}}(ns,\mathcal{G})\geq\frac{\pi(\mathcal{G})N^k}{k!}+o(N^k)=\frac{ex(N,\mathcal{G})}{e(\mathcal{F})}+o(N^k).$$ 
\hfill$\square$

\bigskip

\noindent{\bf Proof of Proposition \ref{prop-chara-upper}}:
The first statement is straightforward. We now prove the second statement.

For sufficiency, assume that there exists a real number $\epsilon>0$ such that for infinitely many $n$, every graph $\mathcal{H}_n\in EX_k(n,\mathcal{G})$ satisfies 
\begin{align}\label{ineq-prop-2}
\nu_{\mathcal{H}_n}(\mathcal{F})\leq e(\mathcal{F})^{-1}ex_k(n,\mathcal{G})-\epsilon n^k.
\end{align}
We show that $ex_\mathcal{F}(n,\mathcal{G})$ does not attain the upper bound in Ineq.~(\ref{lower-upper}).
For each integer $n$ satisfying Ineq.~(\ref{ineq-prop-2}), let $\mathcal{R}_n$ be an $n$-vertex rainbow $\mathcal{G}$-free $k$-uniform hypergraph consisting of $ex_\mathcal{F}(n,\mathcal{G})$ $\mathcal{F}$-copies.
Assume toward a contradiction that
$ex_\mathcal{F}(n,\mathcal{G})= e(\mathcal{F})^{-1}ex_k(n,\mathcal{G})+o(n^k)$. Then 
$e(\mathcal{R}_n)=ex_k(n,\mathcal{G})+o(n^k)$.
By Lemma \ref{clm-thm1-1}, there are $O(n^{k-1})$ copies of $\mathcal{G}$ in $\mathcal{R}_n$.
By Theorem \ref{hypergraph-supersaturation}, 
we can delete at most $o(n^{k})$ edges from $\mathcal{R}_n$ such that the resulting graph, denoted by $\mathcal{R}'_n$, is $\mathcal{G}$-free.
Since the extremal graphs for $\mathcal{G}$ are stable, there exists an $\mathcal{H}_n\in EX_k(n,\mathcal{G})$ such that  $\mathcal{H}_n$ can be obtained form $\mathcal{R}'_n$ by deleting or adding at most $o(n^k)$ edges.
It follows that $\mathcal{H}_n$ can be obtained form $\mathcal{R}_n$ by modifying at most $o(n^k)$ $\mathcal{F}$-copies. Therefore, $\mathcal{H}_n$ 
contains at least $ex_\mathcal{F}(n,\mathcal{G})-o(n^k)=e(\mathcal{F})^{-1}ex_k(n,\mathcal{G})+o(n^k)$ $\mathcal{F}$-copies, contradicting Ineq.~(\ref{ineq-prop-2}).

For necessity, we assume  that $ex_\mathcal{F}(n,\mathcal{G})$ does not attain the upper bound in Ineq.~(\ref{lower-upper}).
In other words, there exists $\epsilon>0$ such that $ex_\mathcal{F}(n,\mathcal{G})\leq e(\mathcal{F})^{-1}ex_k(n,\mathcal{G})-\epsilon n^k$ for infinitely many $n$.
This implies that for each such $n$, every $\mathcal{H}_n\in EX_k(n,\mathcal{G})$ satisfies $\nu_{\mathcal{H}_n}(\mathcal{F})\leq e(\mathcal{F})^{-1}ex_k(n,\mathcal{G})-\epsilon n^k.$ 
\hfill$\square$

\bigskip

\noindent{\bf Proof of Theorem \ref{thm-fano}}: 
The proof depends on the stability of the extremal graphs of 
$\mathcal{G}\in\{\mathbb{F},\mathcal{C}_3^{(2k)},\mathcal{B}_3^{(3)},\mathcal{B}_4^{(4)}\}$.
We first introduce extremal graphs below.
\begin{enumerate}
    \item Let $\mathcal{A}_n^{(3)}$ denote the $3$-uniform hypergraph obtained by partitioning $n$-vertex set into two of size $\lfloor n/2\rfloor$ and $\lceil n/2\rceil$,  and taking all edges intersecting each parts in nonempty sets.
    It has been proved that $EX_3(n,\mathbb{F})=\{\mathcal{A}_n^{(3)}\}$ for sufficiently large $n$, and the stability theorem exists (see \cite{C-Furedi,Fano-stable}).
    
    \item Let $\mathcal{T}_n^{(3)}$ represents the $3$-uniform hypergraph obtained by partitioning $n$-vertex set into two parts $X,Y$ with $||X|-2n/3|<1$, and taking all edges with two vertices in $X$. 
    It has been proved that $EX_3(n,\mathcal{B}_3^{(3)})$ consists of all such $\mathcal{T}_n^{(3)}$ for sufficiently large $n$, and the stability theorem exists (see \cite{F32-1,F32-2,F32-3}).
    
    \item Let $\mathcal{B}^{(2k)}(n,t)$ denote the $2k$-uniform hypergraph obtained by partitioning an $n$-vertex set into two parts with size $n/2-t$ and $n/2+t$, and taking as  all $2k$-tuples with odd intersecting with each part. Let $\mathscr{B}_n^{(2k)}$ denote the set of such graphs $\mathcal{B}^{(2k)}(n,t)$ of maximum edges.
    It has been proved that $EX_{2k}(n,\mathcal{C}_3^{(2k)})=\mathscr{B}_n^{(2k)}$ and $EX_{4}(n,\mathcal{B}_4^{(4)})=\mathscr{B}_n^{(4)}$ for sufficiently large $n$, and the stability theorems exist (see \cite{C3-2k,B44}).
\end{enumerate}

If $\mathcal{F}$ is a $3$-uniform hypergraph and $\mathcal{F}$ is not bipartite, then $\mathcal{A}_n^{(3)}$ does not contain $\mathcal{F}$ as a subgraph for any $n$.
By the second statement in Proposition \ref{prop-chara-upper}, $ex_{\mathcal{F}}(n,\mathbb{F})$ can not attain the upper bound in Ineq.~(\ref{lower-upper}).
If $\mathcal{F}$ is a $3$-uniform hypergraph such that for every independent set $S$ of $\mathcal{F}$, $\mathcal{F}-S$ is not an empty graph, then $\mathcal{T}_n^{(3)}$ does not contain $\mathcal{F}$ as a subgraph for any $n$, and hence $ex_{\mathcal{F}}(n,\mathcal{B}_3^{(3)})$ can not attain the upper bound in Ineq.~(\ref{lower-upper}).
If $\mathcal{F}$ is a $2k$-uniform hypergraph and $\mathcal{F}$ is not odd, then every graph in $\mathscr{B}_n^{(2k)}$ does not contain $\mathcal{F}$ as a subgraph for any $n$. Hence, $ex_{\mathcal{F}}(n,\mathcal{C}_3^{(2k)})$ and $ex_{\mathcal{F}}(n,\mathcal{B}_4^{(4)})$ can not attain the upper bound in Ineq.~(\ref{lower-upper}).

If $\mathcal{F}$ is a bipartite $3$-uniform hypergraph, then each edge in $\mathcal{A}_n^{(3)}$ (for even $n$) is contained in the same number of copies of $\mathcal{F}$.
By Lemma \ref{lem-two-packings} and the second statement in Proposition \ref{prop-chara-upper}, we have that $ex_{\mathcal{F}}(n,\mathbb{F})$ attains the upper bound in Ineq.~(\ref{lower-upper}).
If $\mathcal{F}$ is a $3$-uniform hypergraph that contains an independent set $S$ such that every edge of $\mathcal{F}$ has nonempty intersection with $S$, then every edge in $\mathcal{T}_n^{(3)}$ is contained in the same number of copies of $\mathcal{F}$.
Hence, by Lemma \ref{lem-two-packings} and the second statement in Proposition \ref{prop-chara-upper}, $ex_{\mathcal{F}}(n,\mathcal{B}_3^{(3)})$ attains the upper bound in Ineq.~(\ref{lower-upper}).

Next, we show that if $\mathcal{F}$ is an odd $2k$-uniform hypergraph, then $ex_{\mathcal{F}}(n,\mathcal{C}_3^{(2k)})$ and $ex_{\mathcal{F}}(n,\mathcal{B}_4^{(4)})$ attain the upper bound in Ineq.~(\ref{lower-upper}).
It is suffices to show that every graph in $\mathscr{B}_n^{(2k)}$ contains $e(\mathcal{F})^{-1}ex_{2k}(n,\mathcal{G})+o(n^{2k})$ $\mathcal{F}$-copies.
Keevash and Sudakov \cite{C3-2k} showed that for each $\mathcal{B}(n,t^*)\in \mathscr{B}_n^{(2k)}$, $0\leq t^*<\sqrt{2kn}$. 
Let $X,Y$ be two parts of $\mathcal{B}(n,t^*)$ with $|X|=n/2-t^*$ and $|Y|=n/2+t^*$.
Let $Y'$ be a subset of $Y$ with $|Y-Y'|=|X|$.
Then $|Y'|<\sqrt{8kn}$.
Then 
$$e(\mathcal{B}(n,t^*)-Y')\geq e(\mathcal{B}(n,t^*))-\left[{n \choose 2k}-{(n-|Y'|)\choose 2k}\right]=e(\mathcal{B}(n,t^*))-o(n^{2k}).$$
Since $|X|=|Y-Y'|$, each edge in $e(\mathcal{B}(n,t^*)-Y')$ is contained in the same number of copies of $\mathcal{F}$.
By Lemma \ref{lem-two-packings}, 
$$\nu_{\mathcal{B}(n,t^*)}(\mathcal{F})\geq \nu_{\mathcal{B}(n,t^*)-Y'}(\mathcal{F})=\frac{e(\mathcal{B}(n,t^*)-Y')}{e(\mathcal{F})}-o(n^{2k})=\frac{e(\mathcal{B}(n,t^*))}{e(\mathcal{F})}-o(n^{2k}).$$
Therefore, there are 
$e(\mathcal{F})^{-1}e(\mathcal{B}(n,t^*))-o(n^{2k})$ $\mathcal{F}$-copies in every $\mathcal{B}(n,t^*)\in \mathscr{B}_n^{(2k)}$.
\hfill$\square$

\bigskip 

\noindent{\bf Proof of Corollary \ref{coro-refine-simple}}: 
The sufficiency holds due to Eq. \eqref{eq-intro-1}.
We now give a short proof of necessity.
Assume that $\chi(G)=r+1\geq 3$.
For any $n$, let $t=\lfloor n/r\rfloor$ and $n'=rt$.
For any graph $F$ with $\chi(F)\leq r$, observe that every edge of $T_{n',r}$ lies in the same number of copies of $F$. By Lemma \ref{lem-two-packings}, the Tur\'{a}n graph $T_{n',r}$ contains $\frac{e(T_{n',r})}{e(F)}+o(n^2)$ $F$-copies.
Since $T_{n',r}$ is $G$-free, it follows that 
$$ex_F(n,G)\geq \frac{e(T_{n',r})}{e(F)}+o(n^2)=\frac{(\chi(G)-2)n^2}{2e(F)(\chi(G)-1)}+o(n^2).$$
Ineq. (\ref{lower-upper}) implies that equality holds in the above expression.
\hfill$\square$




\begin{thebibliography}{1}
\bibitem{Alon-Shapira} N. Alon, A. Shapira, On an extremal hypergraph problem of Brown, Erd\H{o}s and S\'{o}s. Combinatorica, 26(6): 2006, 627--646.

\bibitem{C5-C3-2}    J. Balogh, A. Liebenau, L. Mattos, N. Morrison, On multicolor Tur\'{a}n numbers, SIAM Journal of Discrete Mathematics, 38(3), 2024, 2297--2311.

\bibitem{C-Furedi}  D. De Caen, Z. F\"{u}redi, The maximum size of $3$-uniform hypergraphs not containing a Fano plane, J. Combinatorial Theory B, 78, 2000, 274--276.


\bibitem{BROWN} W.G. Brown, P. Erd\H{o}s, V.T. S\'{o}s, Some extremal problems on $r$-graphs. New Directions in the Theory of Graphs, Proc. 3rd Ann Arbor Conference on Graph Theory, Academic Press, New York, 1973, 55--63.

\bibitem{BROWN-2} W.G. Brown, P. Erd\H{o}s, V.T. S\'{o}s, On the existence of triangulated spheres in 3-graphs and related problems, Periodica Mathematica Hungaria, 3, 1973, 221--228.

    
\bibitem{Erdos-Sim} P. Erd\H{o}s, M. Simonovits, Supersaturated graphs and hypergraphs, Combinatorica, 3(2), 1983, 181--192.
    




\bibitem{B44} Z. F\"{u}redi, O. Pikhurko, Quadruple systems with independent neighboors, J. Combinatorial Theory B, 115, 2008, 1552--1560.

\bibitem{F32-1} Z. F\"{u}redi, O. Pikhurko, M. Simonovits, On triple systems with independent neighborhoods. Combin. Probab. Comput.,  14, 2005, 795--813.

\bibitem{F32-2} Z. F\"{u}redi, O. Pikhurko, M. Simonovits, The Tur\'{a}n density of the hypergraph $\{abc,ade,bde,cde\}$. Electron. J. Combin.,  10, 2003, 8 pp.


\bibitem{Gowers} W.T. Gowers, B. Janzer, Generalizations of the Ruzsa-Szemer\'{e}di and rainbow Tur\'{a}n problems for cliques, Combinatorics, Probability and Computing, 30(4): 2021, 591--608.


\bibitem{fractional-graph} P.E. Haxell, V. R\"{o}dl, Integer and fractional packings in dense graphs, Combinatorica, 21(1), 2001, 13--38.
    
\bibitem{IKNV} A. Imolay, J. Karl, Z.L. Nagy, B. V\'{a}li, Multicolor Tur\'{a}n numbers, Discrete Mathematics, 345, 2022, 112976.

\bibitem{Fano-stable} P. Keevash, B. Sudakov, The Tur\'{a}n number of the Fano plane. Combinatorica, 25(5), 2005, 561--574.

\bibitem{C3-2k} P. Keevash, B. Sudakov, On a hypergraph Tur\'{a}n problem of Frankl. Combinatorica, 25(6), 2005, 673--706.
    
\bibitem{C5-C3-1} B. Kov\'{a}cs, Z.L. Nagy, Multicolor Tur\'{a}n numbers II: A generalization of the Ruzsa-Szemer\'{e}di theorem and new results on cliques and odd cycles. Journal of Graph Theory, 107(3), 2024, 629--641.

\bibitem{F32-3} D. Mubayi, V. R\"{o}dl, On the Tur\'{a}n number of triple systems. J. Combin. Theory Ser. A., 100(1), 2002, 135--152.
    
\bibitem{NRS} B. Nagle, V. R\"{o}dl, and M. Schacht, The counting lemma for regular $k$-uniform hypergraphs, Random Structures and Algorithms, 28(2), 2006, 113--179. 
    
\bibitem{RSST} V. R\"{o}dl, M. Schacht, M. Siggers, and N. Tokushige, Integer and fractional packings of hypergraphs, Journal of Combinatorial Theory, Series B, 97, 2007, 245--268.
    
\bibitem{RJ} V. R\"{o}dl, J. Skokan, Regularity lemma for uniform hypergraphs, Random Structures
 and Algorithms, 25(1), 2004, 1--42.
 

    
\bibitem{RS}   V. R\"{o}dl, J. Skokan, Applications of the regularity lemma for uniform hypergraphs. Random Structures and Algorithms, 28(2), 2006, 180--194.
    

    
\bibitem{6-3} I.Z. Ruzsa, E. Szemer\'{e}di, Triple systems with no six points
carrying three triangles,
in: Combinatorics, Keszthely, 1976, Collog. Math. Soc. l\'{a}nos Bolvai, 18, 1978, 939--945.
    


\bibitem{PHD}  M. Schacht, On the regularity method for hypergraphs, Ph.D. thesis, Emory University, 2004. 
    
\bibitem{simple-reg}  E. Szemer\'{e}di, Regular partitions of graphs. In Probl\'{e}mes combinatoires et th\'{e}orie des graphes ICollog. Internat. CNRS, Univ. Orsay, Orsay, 1976), volume 260 of Collog. Internat. CNRspages 399-401. CNRS, Paris, 1978. 

\end{thebibliography}
\end{document}